\documentclass[suppldata]{interact}
 \usepackage{hyperref}
\usepackage{epstopdf}
\usepackage[caption=false]{subfig}
\usepackage[numbers,sort&compress]{natbib}
\bibpunct[, ]{[}{]}{,}{n}{,}{,}
\makeatletter
\def\NAT@def@citea{\def\@citea{\NAT@separator}}
\makeatother

\theoremstyle{plain}
\newtheorem{theorem}{Theorem}[section]
\newtheorem{lemma}[theorem]{Lemma}
\newtheorem{corollary}[theorem]{Corollary}
\newtheorem{proposition}[theorem]{Proposition}

\theoremstyle{definition}

\newtheorem{example}[theorem]{Example}

\theoremstyle{remark}
\newtheorem{remark}{Remark}

\renewcommand{\c}{\text{Cos}}
\renewcommand{\S}{\text{Sin}}
\begin{document}
\title{A Lagrangian  method for indefinite $q$-integrals }
\author{
\name{Gamela E. Heragy\textsuperscript{a}\thanks{CONTACT Gamela E. Heragy. Email: moonegam123@gmail.com}, Zeinab S.I. Mansour, \textsuperscript{b}\thanks{CONTACT Zeinab S.I. Mansour. Email: zsmansour@cu.edu.eg} and Karima M. Oraby\textsuperscript{a}\thanks{CONTACT Karima M. Oraby. Email: koraby83@yahoo.com}}
\affil{\textsuperscript{a}Mathematics Department, Faculty of Science, Suez University, Suez, Egypt.
 \\ \textsuperscript{b}Mathematics Department, Faculty of Science, Cairo University, Giza, Egypt.
}
}
\maketitle
\begin{abstract}
A Lagrangian  method  is introduced recently for deriving indefinite integrals of special functions that satisfy homogeneous (nonhomogeneous) second-order linear differential equations.  This paper extends this method to include indefinite Jackson $q$-integrals of special functions satisfying homogeneous (nonhomogeneous) second-order linear $q$-difference equations. Many $q$-integrals, both previously known and completely new, are derived using the method.  We introduce samples of indefinite and definite $q$-integrals for Jackson's $q$-Bessel functions, $q$-hypergeometric functions, and some orthogonal polynomials.
\end{abstract}

\section{Introduction and Preliminaries }
Conway in \cite{conway}  introduced a simple method of deriving indefinite integrals. The method applies to any special function satisfying an ordinary differential equation. The main result derived in \cite{conway} is the indefinite integral
\begin{align*}
 & \int f(x) \left( h^{''}(x) +p(x) h^{'}(x) + q(x) h(x)\right)y(x) = f(x) \left(h^{'}(x) y(x) -h(x) y^{'}(x)\right),
\end{align*}
where $f(x)$ satisfies the first-order differential equation
\begin{align*}
  f(x) = \exp (\int p(x) dx).
\end{align*}
Here $p(x)$ and
$q(x)$ are arbitrary complex-valued differentiable functions of $x$  in $\mathbb{R}$, with $h(x)$ being at least twice differentiable. In a series of papers, see \cite{conway,conwa,conw,con,cpbg,co,c1,c2},   Conway developed this method to obtain more indefinite integrals.  In this paper, we extend Conway's results to include special functions satisfying   second-order $q$-difference equations.

This paper is organized as follows. In the rest of this section, we introduce the notion and notations from the $q$-series needed in the sequel.  In Section 2,  we derive $q$-analogs of the Euler-Lagrange method to homogenous second-order $q$-difference equations.   Section 3 contains applications of the derived method to Jackson's $q$-Bessel functions and some other special functions. Finally, in Section 4, we extend the Euler-Lagrange method to nonhomogeneous second-order $q$-difference equations.\\
Throughout this paper, $q$ is a positive number less than 1, $\mathbb{N}$ is the set of positive integers,  and $\mathbb{N}_0$ is the set of non-negative integers.
We follow Gasper and Rahman \cite{rahman} for the definition of the $q$-shifted factorial,  $q$-gamma, $q$-beta function,  and  $q$-hypergeometric series.\\
A $q$-natural number $[n]_q$ is defined by $ [n]_q=\frac{1-q^n}{1-q},\,\,n\in\mathbb N_0$.
The $q$-derivative $D_qf(x)$ of a function $f$ is defined by \cite{j, hein}
\[(D_qf)(x)=\frac{f(x)-f(qx)}{(1-q)x},\,\,if \, x\neq0,\]
and $(D_qf)(0)= f'(0)$ provided $f'(0)$ exists.
Jackson's $q$-integral  of a function $f$   is defined by \cite{jackson}
\begin{align}\label{nun}
  \int_0^a f(t)d_qt:=(1-q)a\sum_{n=0}^\infty q^n f(aq^n) ,\,\,a\in\mathbb R,
\end{align}
provided that the corresponding series in ({\ref{nun}})  converges.
Jackson also introduced   three $q$-analogs of Bessel functions, \cite{jackson,rahman}, they  are defined by
\begin{align*}
  J_{\nu}^{(1)}(z;q)= \frac{(q^{v+1};q)_\infty}{(q;q)_\infty} \sum_{n=0}^{\infty} \dfrac{(-1)^n }{(q,q^{v+1};q)_n} (z/2)^{2n+\nu},\quad |z|<2,
\end{align*}
\begin{align*}
  J_{\nu}^{(2)}(z;q)= \frac{(q^{v+1};q)_\infty}{(q;q)_\infty} \sum_{n=0}^{\infty} \dfrac{(-1)^n q^{n(n+\nu)} }{(q,q^{v+1};q)_n} (z/2)^{2n+\nu},\quad z\in \mathbb C,
\end{align*}
\begin{align*}
  J_{\nu}^{(3)}(z;q)= \frac{(q^{v+1};q)_\infty}{(q;q)_\infty} \sum_{n=0}^{\infty} \dfrac{(-1)^n q^{\frac{n(n+1)}{2}} }{(q,q^{v+1};q)_n} (z)^{2n+\nu},\quad z\in \mathbb C.
\end{align*}
Hahn \cite{gfd} found that
\begin{align*}
  J_\nu^{(1)} (z;q) = J_{\nu}^{(2)}(z;q)/ (-z^2/4;q^2)_\infty, \mid z\mid  < 2, \nu > -1.
\end{align*}
We shall use the simpler notation
\begin{align*}
   J_{\nu}^{(2)}( \lambda x |q^2)&=J_{\nu}^{(2)}(2\lambda x(1-q);q^2).
\end{align*}
There are three known $q$-analogs of the trigonometric  functions, \{$\sin_q z$,$\cos_q z$\}, \{$\S_q z$,$\c_q z$\} and \{$\sin(z;q)$,$\cos(z;q)$\}. Each set of $q$-analogs is related to one of the three $q$-analogs of Bessel functions.\\
The functions $\sin_q z$ and $\cos_q z$ are defined for $|z| < \frac{1}{1-q}$ by
\begin{align*}
\sin_q z &=\frac{(q^2;q^2)_\infty}{(q;q^2)_\infty} (z)^{1/2} J^{(1)}_{1/2} (2 z;q^2) = \sum_{n=0}^{\infty} (-1)^n  \frac{z^{2n+1}}{[2n+1]_q!},
\end{align*}
\begin{align*}
\cos_q z &=\frac{(q^2;q^2)_\infty}{(q;q^2)_\infty} (z)^{1/2} J^{(1)}_{-1/2} (2 z;q^2)  = \sum_{n=0}^{\infty} (-1)^n  \frac{z^{2n}}{[2n]_q!}.
\end{align*}
The functions $\S_q z$ and $\c_q z$ are defined for $ z \in \mathbb C$ by
\begin{align*}
\S_q z&=\frac{(q^2;q^2)_\infty}{(q;q^2)_\infty} (z)^{1/2} J^{(2)}_{1/2} (2 z;q^2)  = \sum_{n=0}^{\infty} (-1)^n  \frac{q^{2n^2+n} z^{2n+1}}{[2n+1]_q!},
\end{align*}
\begin{align*}
\c_q z&=\frac{(q^2;q^2)_\infty}{(q;q^2)_\infty} (z)^{1/2} J^{(2)}_{-1/2} (2 z;q^2)  = \sum_{n=0}^{\infty} (-1)^n  \frac{q^{2n^2-n} z^{2n}}{[2n]_q!}.
\end{align*}
Finally, the functions $\sin(z;q)$ and $\cos(z;q)$ are defined for $ z \in \mathbb C$ by
\begin{align*}
\sin(z;q)&=\Gamma_q (1/2) (z(1-q))^{1/2} J^{(3)}_{1/2} (z(1-q);q^2) = \sum_{n=0}^{\infty} (-1)^n \frac{q^{n^2+n} z^{2n+1}}{\Gamma_q (2n+2)},
\end{align*}
\begin{align*}
\cos(z;q)&=\Gamma_q (1/2) (z q ^{-1/2}(1-q))^{1/2} J^{(3)}_{-1/2} (z(1-q)/\sqrt{q};q^2)  = \sum_{n=0}^{\infty} (-1)^n  \frac{q^{n^2} z^{2n}}{\Gamma_q (2n+1)}.
\end{align*}
The $q$-trigonometric  functions satisfy the $q$-difference equations
\begin{align*}
  D_q \sin_q z =  \cos_q (z), \quad \quad D_q  \cos_q z = - \sin_q (z).
\end{align*}
\begin{align*}
  D_q \S_q z =  \c_q (qz), \quad \quad D_q \c_q z = - \S_q (qz).
\end{align*}
\begin{align*}
  D_q \sin( z;q) =  \cos(q^{\frac{1}{2}}  z;q), \quad \quad D_q \cos( z;q) = -q^{\frac{1}{2}} \sin(q^{\frac{1}{2}} z;q).
\end{align*}
The classical Struve Function \cite[P. 328]{Watson} is defined by
\[H_{\nu}(z)=\frac{2(z/2)^{\nu}}{\Gamma(1/2)\Gamma(\nu+\frac12)}\int_{0}^{1} (1-t^2)^{\nu-1/2}\sin zt\,dt,\,\,\Re\nu>-\frac12.\]
It has the series representation
\[H_{\nu}(z)=\sum_{n=0}^\infty\dfrac{(-1)^n(x/2)^{2n+\nu+1}}{\Gamma(n+\frac32)\Gamma(n+\nu+\frac32)},\]
and it is the solution of the non-homogenous equation, see \cite{struve},
\begin{align*}
  x^2 \frac{d^2y}{dx^2} +x \frac{dy}{dx} +(x^2-\nu^2) y(x) = \frac{4(x/2)^{\nu+1}}{\sqrt{\pi} \Gamma(\nu+1/2)}.
\end{align*}
One can verify that  \[\lim_{q\to 1^-}H^{(1)}_{\nu}((1-q)z;q^2)=\lim_{q\to 1^-}H^{(2)}_{\nu}(q^{\nu+\frac12}(1-q)z;q^2)=\lim_{q\to 1^-}H^{(3)}_{\nu}(z;q^2)=H_\nu(z).\]
Oraby and Mansour \cite{struve} introduced three $q$-analogs of the   Bessel-Struve functions, $H^{(k)}_\nu(z;q^2),\,(k=1,\,2,\,3),$  they are defined by
\begin{align*}
  H_\nu^{(1)}(x;q^2)= (1+q) \dfrac{(x/1-q^2)^\nu}{\Gamma_{q^2}(\frac{1}{2})\Gamma_{q^2}(\nu+\frac{1}{2})}\int_{0}^{1} (q^2 t^2;q^2)_{\nu-\frac{1}{2}} \sin_q xt  \, d_qt,
\end{align*}
\begin{align*}
  H_\nu^{(2)}(x;q^2)= q^{-\nu-\frac{1}{2}}(1+q) \dfrac{(q^{-\nu-\frac{1}{2}}x/1-q^2)^\nu}{\Gamma_{q^2}(\frac{1}{2})\Gamma_{q^2}(\nu+\frac{1}{2})}\int_{0}^{1} (q^2 t^2;q^2)_{\nu-\frac{1}{2}} \S_q xt  \, d_qt,
\end{align*}
and
\begin{align*}
  H_\nu^{(3)}(x;q^2)= (1+q) \dfrac{(x/1+q)^\nu}{\Gamma_{q^2}(\frac{1}{2})\Gamma_{q^2}(\nu+\frac{1}{2})}\int_{0}^{1} (q^2 t^2;q^2)_{\nu-\frac{1}{2}} \sin(xt;q) \, d_qt  , x \in \mathbb{C},
\end{align*}
where $ \Re{(\nu)}>- \frac{1}{2}$.

\begin{lemma}
Let $u(x)$, and $v(x)$ be continuous functions at zero. The q-integration by parts rules
\begin{align}\label{555}
  \frac{1}{q} \int_{a}^{b} D_{q^{-1}}u(x) v(x) d_qx &= u(x/q) v(x)\Big | _a^b - \int_{a}^{b} u(x) D_q v(x) d_qx \nonumber  \\&
  = u(x/q) v(x/q)\Big | _a^b - \frac{1}{q} \int_{a}^{b} u(x/q) D_{q^{-1}}v(x) d_qx .
\end{align}
\end{lemma}
\begin{lemma}\label{25454}
 \cite{gfdrf} Let the functions $f$ and $ g$ be defined and continuous on $[0,\infty[$. Assume that the
improper Riemann integrals of the functions $f(x)g(x)$ and $ f(x/q)g(x)$ exist on $[0,\infty[$. Then
\begin{align}\label{21457777}
  \int_{0}^{\infty} f(x) D_q g(x) dx&= \frac{f(0)g(0)}{1-q} \ln q - \frac{1}{q} \int_{0}^{\infty} g(x) D_{q^{-1}} f(x) dx \nonumber\\&= \frac{f(0)g(0)}{1-q} \ln q -  \int_{0}^{\infty} g(qx) D_q f(x) dx.
\end{align}
\end{lemma}
In the following, we use (HSOqDE) to denote the homogenous second-order $q$-difference equation and (NHSOqDE) to denote the non-homogenous second-order $q$-difference equation.
\section{Extensions of Lagrangian Method for HSOqDE }
In this section, we extend the Lagrangian method introduced in \cite{conway} to functions  satisfying  homogenous second-order  $q$-difference equation of the form  ({\ref{mxlk}}) or ({\ref{mx}}) below.
\begin{theorem}\label{bvc}
Let $p(x)$ and $r(x)$ be continuous functions at zero. Let  $y(x)$ be a solution  of the second-order  $q$-difference equation
\begin{equation}\label{mxlk}
  \frac{1}{q}D_{q^{-1}} D_q y(x) +  p(x)D_{q^{-1}}y(x) + r(x) y(x) =0.
\end{equation}
\text{ Then}
\begin{align}\label{mkl}
 & \int f(x) \Big( \frac{1}{q} D_{q^{-1}} D_qh(x)+   p(x) D_{q^{-1}} h(x) +r(x) h(x)\Big) y(x) d_qx \nonumber\\&=
   f(x/q)\Big (y(x/q) D_{q^{-1}}h(x) - h(x/q) D_{q^{-1}}y(x) \Big),
\end{align}
where $ h(x) $ is an arbitrary function,   and $ f(x)$  is a solution of the first order  $q$-difference equation
\begin{equation}\label{fht}
  \frac{1}{q} D_{q^{-1}} f(x) =  p(x) f(x) .
\end{equation}
\end{theorem}
\begin{proof}
  Applying the  $q$-integration by parts rule ({$\ref{555}$})
  \begin{align*}
&\frac{1}{q}\int f(x)y(x)  D_{q^{-1}} D_qh(x) d_qx  =   (D_qh)(x/q)f(x/q) y(x/q)- \frac{1}{q} \int (D_qh)(x/q)D_{q^{-1}}(f(x) y(x)) d_qx.
\end{align*}
Therefore, by the $q$-product rule, we get
\begin{align*}
  \frac{1}{q}\int f(x)y(x)  D_{q^{-1}} D_qh(x) d_qx &= D_{q^{-1}}h(x) f(x/q) y(x/q)- \frac{1}{q} \int (D_qh)(x/q) f(x/q)D_{q^{-1}} y(x)d_qx  \\&-\frac{1}{q} \int (D_qh)(x/q)D_{q^{-1}}(f(x))y(x)d_qx.
\end{align*}
Applying ({$\ref{555}$}), we obtain
\begin{align*}
  \frac{-1}{q} \int (D_qh)(x/q) f(x/q)&D_{q^{-1}} y(x)d_qx = - h(x/q) f(x/q) D_{q^{-1}} y(x)\\&+\int  h(x) f(x) D_q D_{q^{-1}}y(x) d_qx + \int h(x) D_{q^{-1}} y(x) D_q f(x/q)d_qx.
\end{align*}
Using Equations ({$\ref{mxlk}$}) and  ({$\ref{fht}$}) yields
\begin{align*}
 \frac{1}{q}\int f(x)y(x)  D_{q^{-1}} D_qh(x) d_qx &= D_{q^{-1}}h(x) f(x/q) y(x/q) -h(x/q)f(x/q)D_{q^{-1}}y(x)   \\& - \int h(x) f(x) r(x) y(x) d_qx  -\int D_{q^{-1}}h(x)f(x) p(x) y(x)d_qx,
\end{align*}
and  we get the desired result.
\end{proof}
\begin{remark}
  It is worth noting that the right- hand side  of (${\ref{mkl}}$) can be represented as
\begin{align*}
  f(\frac{x}{q}) W_q(y,h)(\frac{x}{q}) = f(\frac{x}{q}) W_{q^{-1}}(y,h)(x),
\end{align*}
and the right- hand side  of ({$\ref{mkrl}$}) is $ f(x) W_{q^{-1}}(y,h)(x) $, where by $ W_q(y,z)(x)$ we mean
\begin{align*}
 W_q(y,z)(x) = y(x) D_q z(x) -z(x) D_q y(x),
\end{align*}
see \cite{sw,risha}.
\end{remark}
\begin{theorem}\label{brevc}
Let $p(x)$ and $r(x)$ be continuous functions at zero. Let  $y(x)$ be any solution  of the second-order  $q$-difference equation
\begin{equation}\label{mx}
  \frac{1}{q}D_{q^{-1}} D_q y(x) +  p(x)D_q y(x) + r(x) y(x) =0.
\end{equation}
\text{ Then}
\begin{align}\label{mkrl}
 & \int f(x) \Big( \frac{1}{q} D_{q^{-1}} D_qh(x)+   p(x) D_q h(x) +r(x) h(x)\Big) y(x) d_qx \nonumber\\&=
   f(x)\Big ( y(x) D_{q^{-1}}h(x) - h(x) D_{q^{-1}}y(x) \Big),
\end{align}
where $ h(x) $ is an arbitrary function  and $ f(x)$  is a solution of the first order  $q$-difference equation
\begin{equation}\label{frht5}
   D_q f(x) =  p(x) f(x).
\end{equation}
\end{theorem}
\begin{proof}
 The proof follows similarity as the proof of Theorem {$\ref{bvc}$} and is omitted.
\end{proof}
\begin{theorem}\label{bvbvnc}
Let $p(x)$ and $r(x)$ be continuous functions at zero. Let  $y(x)$ be any solution  of the second-order  $q$-difference Equation $({\ref{mxlk}})$.
\text{ Then}
\begin{align*}
 & \int_{0}^{\infty} f(x) \Big[ \frac{1}{q} D_{q^{-1}} D_qh(x)+   p(x) D_{q^{-1}} h(x) +r(x) h(x)\Big] y(x) dx \nonumber\\& =\frac{ f(0) \ln q}{1-q} \left(D_qh(0)y(0)-h(0)D_qy(0)\right),
\end{align*}
where $ h(x) $ is an arbitrary function  and $ f(x)$  is a solution of the first order  $q$-difference Equation $({\ref{fht}})$.
\end{theorem}
\begin{proof}
 Applying Lemma {$\ref{25454}$},  we get
  \begin{align*}
&\int_{0}^{\infty} f(x)y(x) D_q D_{q^{-1}} h(x)  dx = \frac{D_q h(0)f(0) y(0)}{1-q} \ln q - \frac{1}{q} \int_{0}^{\infty} D_{q^{-1}} h(x) D_{q^{-1}}(f(x) y(x)) dx.
\end{align*}
From the $q$-product rule, we obtain
\begin{align*}
  &\int_{0}^{\infty} f(x)y(x) D_q D_{q^{-1}} h(x)  dx = \frac{D_q h(0)f(0) y(0)}{1-q} \ln q - \frac{1}{q} \int_{0}^{\infty} D_{q^{-1}} h(x) f(x/q) D_{q^{-1}} y(x) dx \\&-\frac{1}{q} \int_{0}^{\infty} D_{q^{-1}} h(x) y(x) D_{q^{-1}} f(x) dx.
\end{align*}
Applying Lemma {$\ref{25454}$}, we get
\begin{align*}
  - \int_{0}^{\infty} D_q h(x/q) f(x/q) D_{q^{-1}} y(x) dx &= - \frac{ h(0)f(0) D_q y(0)}{1-q} \ln q + \int_{0}^{\infty} h(x) f(x) D_q D_{q^{-1}}y(x) dx \\&+ \int_{0}^{\infty} h(x) D_{q^{-1}} y(x) D_q f(x/q)dx.
\end{align*}
Using Equations ({$\ref{mxlk}$}) and ({$\ref{fht}$}) yields
\begin{align*}
 \frac{1}{q}\int_{0}^{\infty} f(x)y(x)  D_{q^{-1}} D_qh(x) dx &= \frac{D_q h(0)f(0) y(0)}{1-q} \ln q -\frac{ h(0)f(0) D_q y(0)}{1-q} \ln q \\&- \int_{0}^{\infty}D_{q^{-1}}h(x) y(x) f(x) p(x)dx - \int_{0}^{\infty} r(x) h(x) y(x) f(x) dx.
\end{align*}
Thus, we get the desired result.
\end{proof}
\section{Applications of Lagrangian method to HSOqDE}
There are  unlimited number of cases as $h (x)$ is arbitrary. The art of using
Equation  (${\ref{mkl}}$) is to choose $h(x)$ to give  interesting $q$-integrals. In this section, we introduce  applications to Theorem {\ref{bvc}} and Theorem {\ref{brevc}} to define $q$-integrals of $q$-Bessel functions, and some $q$-orthogonal polynomials.
\begin{theorem}\label{thmn1}
Let $m$ and $\nu$ be complex numbers. If
$\Re  (\nu) > -1 $ and $\Re  (m+\nu) > 0$, we get the $q$-integral
\begin{align}\label{dodo}
 &\int x^{m+1} \left( \frac{1}{(1-q)^2} -\dfrac{ [\nu]^2_q - q^{\nu-m} [m]^2_q}{x^2}\right)J_\nu^{(3)}( x ;q^2) d_qx \nonumber \\&= q^{\nu-m }x^{m+1}\left(\frac{ [m]_q }{x}J_\nu^{(3)}( \frac{ x }{q};q^2) -\frac{1}{q} D_{q^{-1}}J_\nu^{(3)}( x;q^2) \right),
\end{align}
or equivalently,
\begin{align}\label{jo}
 &\int x^{m+1} \left(\frac{ 1}{(1-q)^2} -\dfrac{ [\nu]^2_q - q^{\nu-m} [m]^2_q}{x^2}\right)J_\nu^{(3)}( x ;q^2) d_qx \nonumber \\&=q^{\nu-m} x^{m+1}  \left(\frac{([m]_q-[\nu]_q) }{x}J_\nu^{(3)}( \frac{ x}{q};q^2)+\frac{J_{\nu+1}^{(3)}(x;q^2)}{1-q} \right).
\end{align}
In particular,
\begin{align}\label{ghabg}
&\int x^{\nu+1} J_\nu^{(3)}( x ;q^2) d_qx = (1-q) x^{\nu+1}  J_{\nu+1}^{(3)}(x;q^2),
\end{align}
and
\begin{align}\label{bgnmjjsjs}
&\int x^{1-\nu} J_\nu^{(3)}( x ;q^2) d_qx =-(1-q) {\Big(\frac{x}{q}\Big)}^{1-\nu} J_{\nu-1}^{(3)}(\frac{x}{q};q^2).
\end{align}
\end{theorem}
\begin{proof}
The third Jackson $q$-Bessel function  $J_\nu^{(3)}(x;q^2)$  satisfies the second-order $q$-difference equation \cite{sw}
\begin{align}\label{mnbvv}
&\frac{1}{q} D_{q^{-1}}D_q y(x) + \frac{1}{qx} D_{q^{-1}} y(x) +\frac{q^{-v}}{(1-q)^2}\left(1-\frac{(1-q^{\nu})^2}{x^2}\right)y(x)=0.
\end{align}\label{bgvsdd}
  By comparing  Equation (${\ref{mnbvv}}$) with Equation ({$\ref{mxlk}$}),  we obtain
\begin{align}
  p(x) = \frac{1}{qx},\quad \quad r(x) =  \frac{q^{-v}}{(1-q)^2}\left(1-\frac{(1-q^{\nu})^2}{x^2}\right).
\end{align}
Thus $f(x) = x$ is a solution of Equation (${\ref{fht}}$). Also,
Equation (${\ref{mkl}}$) associated with the  third Jackson $q$-Bessel function will be
\begin{align}\label{maj}
  & \int x (L_{q,\nu}^{(3)}h)(x) J_\nu^{(3)}( x ;q^2) d_qx =
   \frac{x}{q}\Big ( J_\nu^{(3)}( \frac{x}{q};q^2) D_{q^{-1}}h(x) - h(x/q) D_{q^{-1}}J_\nu^{(3)}( x;q^2) \Big),
\end{align}
where
\begin{align*}
  (L_{q,\nu}^{(3)}h)(x) =  \frac{1}{q} D_{q^{-1}} D_qh(x)+   \frac{1}{qx} D_{q^{-1}} h(x) + \frac{q^{-v}}{(1-q)^2}\left(1-\frac{(1-q^{\nu})^2}{x^2}\right)h(x).
\end{align*}
Substituting with $ h(x)=x^m$ into Equation (${\ref{maj}}$) yields
\begin{align*}
 &\int x^{m+1} \left( \frac{1}{(1-q)^2} -\dfrac{[\nu]_q^2 - q^{\nu-m} [m]_q^2}{x^2}\right)J_\nu^{(3)}( x ;q^2) d_qx \\&= q^{\nu-m }x^{m+1}\left(\frac{ [m]_q J_\nu^{(3)}( \frac{ x }{q};q^2)}{x}-\frac{1}{q} D_{q^{-1}}J_\nu^{(3)}( x;q^2) \right),
\end{align*}
using the $q$-difference equation \cite[Eq.(3.5)]{lommel2}
\begin{align}\label{4215f}
  D_{q^{-1}} J_\nu^{(3)}(x;q^2)= \frac{q[\nu]_q}{x}J_\nu^{(3)}(\frac{x}{q};q^2)-\frac{q}{(1-q)}J_{\nu+1}^{(3)}(x;q^2),
\end{align}
we obtain
 \begin{align}
 &\int x^{m+1} \left(\frac{ 1}{(1-q)^2} -\dfrac{ [\nu]^2_q - q^{\nu-m} [m]^2_q}{x^2}\right)J_\nu^{(3)}( x ;q^2) d_qx \nonumber\\& =q^{\nu-m} x^{m+1}  \left(\frac{([m]_q-[\nu]_q) }{x}J_\nu^{(3)}( \frac{ x}{q};q^2)+\frac{J_{\nu+1}^{(3)}(x;q^2)}{1-q} \right).
\end{align}
Substituting  with $ m = \nu$, in Equation (${\ref{jo}}$) gives (${\ref{ghabg}}$).
Substituting  with $m=-\nu$,  in Equation (${\ref{jo}}$) yields
\begin{align*}
&\int x^{1-\nu} J_\nu^{(3)}( x ;q^2) d_qx =q^{2\nu}(1-q) x^{1-\nu} \left( J_{\nu+1}^{(3)}(x;q^2)-\frac{q^{-\nu}(1-q^{2\nu})}{x}J_{\nu}^{(3)}(\frac{x}{q};q^2) \right).
\end{align*}
 Applying  \cite[Eq.(2.14)]{lommel2} (with $x$ is replaced by $\frac{x}{q}$  )
\begin{align*}
  J_{\nu+1}^{(3)}(x;q^2) -\frac{q^{-\nu}(1-q^{2\nu})}{x}J_{\nu}^{(3)}(\frac{x}{q};q^2)=-q^{-1-\nu}J_{\nu-1}^{(3)}(\frac{x}{q};q^2),
\end{align*}
we get (${\ref{bgnmjjsjs}}$) and completes the proof of the theorem.
\end{proof}
\begin{remark}
  Equation (${\ref{ghabg}}$) is equivalent to \cite[Eq.(2.10)]{lommel2}  (with $\nu$ is replaced by $\nu+1$ )
\begin{align*}
  D_q \left(x^{\nu+1} J_{\nu+1}^{(3)}( x ;q^2)  \right) = \frac{x^{\nu+1}}{1-q}J_\nu^{(3)}( x ;q^2).
\end{align*}
\end{remark}
\begin{proposition}
 Let $\mu$ and $\nu$ be  a complex number. Assume that $\Re (\mu) > -1 $ and $\Re (\nu) > -1 $. Then
  \begin{align*}
    &\int x \left(\frac{ q^{-\nu}[\nu-\mu]_q}{1-q} +\dfrac{ q^{-\mu} [\mu]^2_q- q^{-\nu} [\nu]^2_q}{x^2}\right)J_\nu^{(3)}( x ;q^2)J_\mu^{(3)}( x ;q^2) d_qx \nonumber \\&= ([\mu]_q-[\nu]_q )J_\nu^{(3)}( \frac{ x}{q};q^2)J_\mu^{(3)}( \frac{ x}{q};q^2)+\frac{x}{1-q}\left(J_{\nu+1}^{(3)}(x;q^2)J_\mu^{(3)}( \frac{ x}{q};q^2)-J_\nu^{(3)}( \frac{ x}{q};q^2)J_{\mu+1}^{(3)}(x;q^2)\right).
  \end{align*}
\end{proposition}
\begin{proof}
  The proof follows by substituting with  $h(x)=J_\mu^{(3)}( x ;q^2)$ in Equation (${\ref{maj}}$).
\end{proof}
\begin{theorem}\label{thmnbs}
Let $\nu$ and $n$ be a complex number with $\Re  (\nu) > -1 $. Set
\begin{align*}
 & \tilde{C}_{n,\nu} := q^{\frac{3n+\nu+2}{2}}(1-q) A_{n,\nu}, \quad \quad \hat{C}_{n,\nu} := q^{\frac{3n+\nu+1}{2}}(1-q) A_{n,\nu}\\&
 D_{n,\nu}(x) := \frac{q^{-n}[n]_q-q^{-n}[\nu]_q}{x}\sin(\frac{q^{\frac{-1}{2}(n+\nu+2)}x}{1-q};q) +\frac{q^{\frac{-1}{2}(n+\nu+2)}}{1-q}\cos(\frac{q^{\frac{-1}{2}(n+\nu+1)}x}{1-q};q),\\&
 \tilde{D}_{n,\nu}(x) :=\frac{q^{-n}[n]_q-q^{-n}[\nu]_q}{x}\cos(\frac{q^{\frac{-1}{2}(n+\nu+2)}x}{1-q};q) -\frac{q^{\frac{-1}{2}(n+\nu+1)}}{1-q}\sin(\frac{q^{\frac{-1}{2}(n+\nu+1)}x}{1-q};q).
  \end{align*}
  Then
\begin{align}\label{2458744}
&\int  \left(  [2n+1]_q  x^n\cos(\frac{q^{\frac{-1}{2}(n+\nu+1)}x}{1-q};q)+x^{n-1} \tilde{C}_{n,\nu}  \sin(\frac{q^{\frac{-1}{2}(n+\nu)}x}{1-q};q)  \right)  J_\nu^{(3)}( x;q^2) d_qx \nonumber \\&=  q^{\frac{3n+\nu+2}{2}}(1-q)x^{n+1} D_{n,\nu}(x) J_\nu^{(3)}( \frac{x}{q};q^2)+q^{\frac{n+\nu+2}{2}}x^{n+1} \sin(\frac{q^{\frac{-1}{2}(n+\nu+2)}x}{1-q};q) J_{\nu+1}^{(3)}( x;q^2),
\end{align}
and
\begin{align}\label{245870}
&\int  \left(x^{n-1}  \hat{C}_{n,\nu} \cos(\frac{q^{\frac{-1}{2}(n+\nu)}x}{1-q};q) -  [2n+1]_q  x^n \sin(\frac{q^{\frac{-1}{2}(n+\nu+1)}x}{1-q};q) \right)  J_\nu^{(3)}( x;q^2) d_qx \nonumber \\&=q^{\frac{3n+\nu+1}{2}}(1-q) x^{n+1} \tilde{D}_{n,\nu}(x)  J_\nu^{(3)}( \frac{x}{q};q^2)+q^{\frac{n+\nu+1}{2}} x^{n+1}\cos(\frac{q^{\frac{-1}{2}(n+\nu+2)}x}{1-q};q) J_{\nu+1}^{(3)}( x;q^2),
\end{align}
where $A_{n,\nu}= q^{-n}[n]^2_q-q^{-\nu}[\nu]^2_q$.
\end{theorem}
\begin{proof}
From Theorem {$\ref{thmn1}$}, we get
 $f(x) = x$ is a solution of Equation (${\ref{fht}}$).
 Substituting with $ h(x)= x^n \sin(\frac{q^{\frac{-1}{2}(n+\nu)}x}{1-q};q)$ into Equation (${\ref{maj}}$),
we get
\begin{align*}
&\int \left(  [2n+1]_q  x^n\cos(\frac{q^{\frac{-1}{2}(n+\nu+1)}x}{1-q};q)+x^{n-1} \tilde{C}_{n,\nu} \sin(\frac{q^{\frac{-1}{2}(n+\nu)}x}{1-q};q)  \right)  J_\nu^{(3)}( x;q^2) d_qx \nonumber \\&= \left( \frac{q^{\frac{n+\nu+2}{2}}(1-q^n)}{x}\sin(\frac{q^{\frac{-1}{2}(n+\nu+2)}x}{1-q};q) +q^n\cos(\frac{q^{\frac{-1}{2}(n+\nu+1)}x}{1-q};q)\right)x^{n+1}J_\nu^{(3)}( \frac{x}{q};q^2)\\&-q^{\frac{n+\nu}{2}}(1-q) x^{n+1} \sin(\frac{q^{\frac{-1}{2}(n+\nu+2)}x}{1-q};q) D_{q^{-1}}J_\nu^{(3)}(  x;q^2).
\end{align*}
Then, from (${\ref{4215f}}$)
we obtain (${\ref{2458744}}$).
Similarly, Equation (${\ref{245870}}$) follows by
substituting with $ h(x)= x^n \cos(\frac{q^{\frac{-1}{2}(n+\nu)}x}{1-q};q)$ into Equation (${\ref{maj}}$).
\end{proof}
\begin{corollary}
For $\Re  (\nu) > -1 $,
  \begin{align*}
 &\int  \left(  \cos(\frac{q^{\frac{-1}{2}(\nu+1)}x}{1-q};q)+x^{-1}\tilde{C}_{0,\nu}  \sin(\frac{q^{\frac{-1}{2}\nu}x}{1-q};q)  \right)  J_\nu^{(3)}( x;q^2) d_qx \nonumber \\&= q^{\frac{\nu+2}{2}}(1-q) x D_{0,\nu}(x) J_\nu^{(3)}( \frac{x}{q};q^2)+q^{\frac{\nu+2}{2}}x \sin(\frac{q^{\frac{-1}{2}(\nu+2)}x}{1-q};q) J_{\nu+1}^{(3)}( x;q^2),
\end{align*}
and
\begin{align*}
&\int  \left(x^{-1}  \hat{C}_{0,\nu}  \cos(\frac{q^{\frac{-1}{2}\nu}x}{1-q};q) - \sin(\frac{q^{\frac{-1}{2}(\nu+1)}x}{1-q};q) \right)  J_\nu^{(3)}( x;q^2) d_qx \nonumber \\&= q^{\frac{\nu+1}{2}}(1-q) x \tilde{D}_{0,\nu}(x)  J_\nu^{(3)}( \frac{x}{q};q^2)+ q^{\frac{\nu+1}{2}} x \cos(\frac{q^{\frac{-1}{2}(\nu+2)}x}{1-q};q) J_{\nu+1}^{(3)}( x;q^2),
\end{align*}
where  $D_{0,\nu}(x)$, $\tilde{D}_{0,\nu}(x)$, $\tilde{C}_{0,\nu}$ and $\hat{C}_{0,\nu}$ are defined as  in Theorem {$\ref{thmnbs}$} with $(n=0)$.
\end{corollary}
\begin{proof}
  The proof follows by setting $n=0$ in Theorem {$\ref{thmnbs}$}.
\end{proof}
\begin{corollary}
For $\Re  (\nu) > -1 $,
  \begin{align*}
&\int x^\nu   \cos(\frac{q^{\frac{-1}{2}(2\nu+1)}x}{1-q};q) J_\nu^{(3)}( x;q^2) d_qx \\& =\frac{x^{\nu+1}}{[2\nu+1]_q}\left( q^{\nu} \cos(\frac{q^{\frac{-1}{2}(2\nu+1)}x}{1-q};q)J_\nu^{(3)}(\frac{x}{q};q^2)+q^{\nu+1}\sin(\frac{q^{-(\nu+1)}x}{1-q};q)J_{\nu+1}^{(3)}( x;q^2)\right),
\end{align*}
and
\begin{align*}
&\int   x^\nu \sin(\frac{q^{\frac{-1}{2}(2\nu+1)}x}{1-q};q) J_\nu^{(3)}( x;q^2) d_qx \\&=\frac{x^{\nu+1}}{[2\nu+1]_q}\left( q^{\nu} \sin(\frac{q^{\frac{-1}{2}(2\nu+1)}x}{1-q};q)J_\nu^{(3)}( \frac{x}{q};q^2)-q^{\nu+\frac{1}{2}}\cos(\frac{q^{-(\nu+1)}x}{1-q};q)J_{\nu+1}^{(3)}( x;q^2)\right).
\end{align*}
\end{corollary}
\begin{proof}
  The proof follows by letting $n=\nu$ in Theorem {$\ref{thmnbs}$}.
\end{proof}
\begin{theorem}\label{thmn2}
For $\Re  (\nu )> -1 $, $\lambda \in \mathbb C $ and $\Re  (m+\nu) >0$, we get the $q$-integral
\begin{align}\label{dodoo}
 &\int \dfrac{x^{m+1}}{(- x^2  \lambda^2(1-q)^2;q^2)_\infty} \left(q^m \lambda^2+ A_{m,\nu} x^{-2}\right) J_\nu^{(2)}( \lambda x |q^2) d_qx = \nonumber\\& \frac{x^{m+1}}{(- x^2  \lambda^2(1-q)^2;q^2)_\infty} \left(\frac{ q^{-m}[m]_q  J_\nu^{(2)}( \lambda x |q^2)}{x}-\frac{1}{q} D_{q^{-1}}  J_\nu^{(2)}( \lambda x |q^2) \right),
\end{align}
or equivalently,
\begin{align}\label{joj}
 &\int \dfrac{x^{m+1}}{(- x^2  \lambda^2(1-q)^2;q^2)_\infty} \left(q^m \lambda^2+A_{m,\nu} x^{-2}\right) J_\nu^{(2)}( \lambda x |q^2) d_qx = \nonumber\\& \frac{x^{m+1}}{(- x^2  \lambda^2(1-q)^2;q^2)_\infty} \left(\frac{ q^{-m}[m]_q-q^{-\nu}[\nu]_q  }{x}J_\nu^{(2)}( \lambda x |q^2)+q^{ \nu } \lambda  J_{\nu+1}^{(2)}( \lambda x |q^2) \right).
\end{align}
In particular,
\begin{align}\label{smsh}
 &\int \dfrac{x^{\nu+1}}{(- x^2  \lambda^2(1-q)^2;q^2)_\infty} J_\nu^{(2)}( \lambda x |q^2) d_qx = \frac{x^{\nu+1}}{\lambda (- x^2  \lambda^2(1-q)^2;q^2)_\infty}  J_{\nu+1}^{(2)}( \lambda x |q^2),
\end{align}
and
\begin{align}\label{smsmn}
 &\int \dfrac{x^{1-\nu}}{(- x^2  \lambda^2(1-q)^2;q^2)_\infty} J_\nu^{(2)}( \lambda x |q^2) d_qx =  \frac{- x^{1-\nu}}{\lambda (- x^2  \lambda^2(1-q)^2;q^2)_\infty} J_{\nu-1}^{(2)}( \lambda x |q^2),
\end{align}
where $A_{m,\nu}$ is defined as in Theorem {$\ref{thmnbs}$}.
\end{theorem}
\begin{proof}
The second Jackson $q$-Bessel function  $ J_\nu^{(2)}( \lambda x |q^2)$  satisfies the second-order $q$-difference equation \cite{ra}
\begin{align}\label{mnbvxx}
&\frac{1}{q} D_{q^{-1}}D_q y(x) + \dfrac{1-q \lambda^2 x^2 (1-q)}{x} D_{q} y(x) +\dfrac{q  \lambda ^{2} x^2  - q^{1-v} [v]^2_q }{x^2} y(x)=0.
\end{align}
  By comparing  Equation (${\ref{mnbvxx}}$) with Equation (${\ref{mx}}$),  we obtain
\begin{align*}
  p(x) = \dfrac{1-q \lambda^2 x^2 (1-q)}{x},\quad  r(x) =  \dfrac{q  \lambda ^{2} x^2  - q^{1-v} [v]^2_q }{x^2}.
\end{align*}
Thus $  f(x) = \dfrac{x}{(- x^2  \lambda^2(1-q)^2;q^2)_\infty} $
 is a solution of Equation (${\ref{frht5}}$).
Also, Equation (${\ref{mkrl}}$) associated with the second Jackson $q$-Bessel function will be
\begin{align}\label{monh}
  & \int \dfrac{x}{(- x^2  \lambda^2(1-q)^2;q^2)_\infty} (L_{\nu,q}^{(2)}h)(x) J_\nu^{(2)}( \lambda x |q^2) d_qx = \nonumber \\&
   \dfrac{x}{(- x^2  \lambda^2(1-q)^2;q^2)_\infty}\Big (J_\nu^{(2)}( \lambda x |q^2) D_{q^{-1}}h(x) - h(x) D_{q^{-1}}J_\nu^{(2)}( \lambda x |q^2) \Big),
\end{align}
where
\begin{align*}
  (L_{\nu,q}^{(2)}h)(x) =  \frac{1}{q} D_{q^{-1}} D_qh(x)+   \dfrac{1-q \lambda^2 x^2 (1-q)}{x} D_{q} h(x) + \dfrac{q  \lambda ^{2} x^2  - q^{1-v} [v]^2_q }{x^2} h(x).
\end{align*}
Substituting with $ h(x)=x^m$
into Equation (${\ref{monh}}$), we get (${\ref{dodoo}}$).
The proof of (${\ref{joj}}$) follows by
using  \cite[Eq.(2.1)]{ra} to obtain
\begin{align}\label{equ25}
  D_{q^{-1}} J_\nu^{(2)}( \lambda x |q^2)= \frac{q^{1-\nu}[\nu]_q}{x} J_\nu^{(2)}( \lambda x |q^2) - \lambda  q^{1+\nu}J_{\nu+1}^{(2)}( \lambda x |q^2),
\end{align}
and by substituting in  Equation (${\ref{dodoo}}$).
Substituting  with $\nu = m $ into Equation (${\ref{joj}}$), yields (${\ref{smsh}}$).
Substituting  with $m=-\nu$,  in Equation (${\ref{joj}}$) and using
\cite[Eq.(1.25)]{rahman} (with $x$ is replaced by $2x \lambda(1-q) $ and $q$ by $q^2$ )
\begin{align}\label{254fdff}
  q^{2\nu} J_{\nu+1}^{(2)}( \lambda x |q^2) = \frac{[2\nu]_q}{\lambda x}J_\nu^{(2)}( \lambda x |q^2) -J_{\nu-1}^{(2)}( \lambda x |q^2),
\end{align}
we get (${\ref{smsmn}}$) and completes the proof.
\end{proof}
\begin{example}
  For $\nu=0$ and $ m = 1$,  Equation (${\ref{joj}}$) will be
\begin{align*}
  &\int \dfrac{\lambda^2 q^2 x^2+1}{(- x^2  \lambda^2(1-q)^2;q^2)_\infty} J_0^{(2)}( \lambda x |q^2) d_qx = \frac{x J_0^{(2)}( \lambda x |q^2)}{(- x^2  \lambda^2(1-q)^2;q^2)_\infty}+  \frac{q \lambda x^2 J_{1}^{(2)}( \lambda x |q^2)}{(- x^2  \lambda^2(1-q)^2;q^2)_\infty}.
\end{align*}
Moreover, if $\lambda$ is a zero of $J_0^{(2)}( x |q^2)$, then
\begin{align*}
  &\int_{0}^{1} \dfrac{\lambda^2 q^2 x^2+1}{(- x^2  \lambda^2(1-q)^2;q^2)_\infty} J_0^{(2)}(  \lambda x |q^2) d_qx = \frac{q \lambda J_{1}^{(2)}( \lambda |q^2)}{(-  \lambda^2(1-q)^2;q^2)_\infty}.
\end{align*}
\end{example}
\begin{theorem}\label{2514b4}
Let $\nu$ and $n$  be a complex number with $\Re  (\nu ) > -1 $.  Let
 \begin{align*}
 &  C_{n,\nu} := q^{\frac{1}{2}} [n+\nu]_q+q^{-\frac{1}{2}} [n-\nu+1]_q,\\&
 I_{n,\nu}(x) := \frac{q^{-n}[n]_q-q^{-\nu}[\nu]_q}{x}\S_q {(q^{n+\frac{1}{2}}x)} +q^{-\frac{1}{2}}\c_q(q^{n+\frac{1}{2}}x),\\&
 \tilde{I}_{n,\nu}(x) := \frac{q^{-n}[n]_q-q^{-\nu}[\nu]_q}{x}\c_q (q^{n+\frac{1}{2}}x) -q^{-\frac{1}{2}}\S_q(q^{n+\frac{1}{2}}x) .
 \end{align*}
 Then
  \begin{align}\label{fdscss}
   &\int\frac{1}{(- x^2  (1-q)^2;q^2)_\infty} \Big(x^{n-1} A_{n,\nu}    \S_q(q^{n+\frac{3}{2}}x)  +x^n C_{n,\nu}(x)  \c_q(q^{n+\frac{3}{2}}x) \Big)  J_\nu^{(2)}( x |q^2) d_qx \nonumber\\&=\frac{x^{n+1}}{(- x^2  (1-q)^2;q^2)_\infty}\Big( I_{n,\nu}(x) J_\nu^{(2)}(  x |q^2)+q^{\nu}  \S_q (q^{n+\frac{1}{2}}x) J_{\nu+1}^{(2)}(  x |q^2)\Big),
  \end{align}
  and
  \begin{align}\label{fdscshnj}
   &\int\frac{1}{(- x^2  (1-q)^2;q^2)_\infty} \Big(x^{n-1}  A_{n,\nu}     \c_q(q^{n+\frac{3}{2}}x)  - x^n C_{n,\nu}(x)  \S_q(q^{n+\frac{3}{2}}x)  \Big)  J_\nu^{(2)}(  x |q^2) d_qx \nonumber \\&=\frac{x^{n+1}}{(- x^2  (1-q)^2;q^2)_\infty}\Big( \tilde{I}_{n,\nu}(x) J_\nu^{(2)}(  x |q^2)+q^{\nu}  \c_q (q^{n+\frac{1}{2}}x) J_{\nu+1}^{(2)}(  x |q^2)\Big),
  \end{align}
  where $A_{n,\nu}$ is defined as in Theorem {$\ref{thmnbs}$}.
\end{theorem}
\begin{proof}
From Theorem {\ref{thmn2}}, we get $  f(x) = \dfrac{x}{(- x^2  \lambda^2(1-q)^2;q^2)_\infty} $
 is a solution of Equation (${\ref{frht5}}$).
  Substitute  with $ h(x)=x^n \S_q (q^{n+\frac{1}{2}}x)$ and $\lambda=1$,
into Equation (${\ref{monh}}$). Using Equation (${\ref{equ25}}$) and  the identity
\begin{align*}
\S_q(q^{n+\frac{1}{2}}x) = q^{n+\frac{1}{2}}(1-q)x \c_q(q^{n+\frac{3}{2}}x)+\S_q(q^{n+\frac{3}{2}}x),
\end{align*}
yields (${\ref{fdscss}}$).
  Similarly, we get the $q$-integral (${\ref{fdscshnj}}$) by
substituting with $ h(x)= x^n \c_q(q^{n+\frac{1}{2}}x)$ into Equation (${\ref{monh}}$).
\end{proof}
\begin{corollary}
 For $\Re (\nu) > -1 $,
  \begin{align*}
   &\int\frac{1}{(- x^2  (1-q)^2;q^2)_\infty} \Big(x^{-1} A_{0,\nu}    \S_q(q^{\frac{3}{2}}x)  + C_{0,\nu}(x)  \c_q(q^{\frac{3}{2}}x) \Big)  J_\nu^{(2)}( x |q^2) d_qx \nonumber\\&=\frac{x}{(- x^2  (1-q)^2;q^2)_\infty}\Big( I_{0,\nu}(x) J_\nu^{(2)}(  x |q^2)+q^{\nu}  \S_q (q^{\frac{1}{2}}x) J_{\nu+1}^{(2)}(  x |q^2)\Big),
  \end{align*}
  and
  \begin{align*}
   &\int\frac{1}{(- x^2  (1-q)^2;q^2)_\infty} \Big(x^{-1}  A_{0,\nu}     \c_q(q^{\frac{3}{2}}x)  -C_{0,\nu}(x)  \S_q(q^{\frac{3}{2}}x)  \Big)  J_\nu^{(2)}(  x |q^2) d_qx \nonumber \\&=\frac{x}{(- x^2  (1-q)^2;q^2)_\infty}\Big( \tilde{I}_{0,\nu}(x) J_\nu^{(2)}(  x |q^2)+q^{\nu}  \c_q (q^{\frac{1}{2}}x) J_{\nu+1}^{(2)}(  x |q^2)\Big),
  \end{align*}
  where  $C_{0,\nu}(x)$, $I_{0,\nu}(x)$ and $\tilde{I}_{0,\nu}(x)$ are defined as in Theorem {$\ref{2514b4}$} with $n=0$ and
  $ A_{0,\nu}$ is defined as  in Theorem {$\ref{thmnbs}$} with $n=0$.
  \end{corollary}
    \begin{proof}
  The proof follows by setting $n=0$ in Theorem {\ref{2514b4}}.
\end{proof}
\begin{corollary}
For $\Re  (\nu) > -1 $,
  \begin{align*}
&\int\frac{x^\nu \c_q(q^{\nu+\frac{3}{2}}x)}{(- x^2  (1-q)^2;q^2)_\infty}      J_\nu^{(2)}(  x |q^2) d_qx \\&=\frac{x^{\nu+1}}{[2\nu+1]_q(- x^2  (1-q)^2;q^2)_\infty}\Big(  \c_q(q^{\nu+\frac{1}{2}}x) J_\nu^{(2)}( x |q^2)+q^{\frac{1}{2}+\nu} \S_q (q^{\nu+\frac{1}{2}}x) J_{\nu+1}^{(2)}(  x |q^2)\Big),
\end{align*}
and
\begin{align*}
&\int\frac{x^\nu \S_q(q^{\nu+\frac{3}{2}}x)}{(- x^2  (1-q)^2;q^2)_\infty} J_\nu^{(2)}(  x |q^2) d_qx \\&=\frac{x^{\nu+1}}{[2\nu+1]_q(- x^2  (1-q)^2;q^2)_\infty}\Big(\S_q(q^{\nu+\frac{1}{2}}x) J_\nu^{(2)}(  x |q^2)-q^{\frac{1}{2}+\nu}  \c_q (q^{\nu+\frac{1}{2}}x) J_{\nu+1}^{(2)}(  x |q^2)\Big).
\end{align*}
\end{corollary}
 \begin{proof}
  The proof follows by setting $ n=\nu$ in Theorem {$\ref{2514b4}$}.
\end{proof}
\begin{theorem}\label{thmjdf}
Let $\mu$ and $\nu$ be complex numbers. Assume that
   $\Re (\nu ) > -1 $ and $ \Re  (m+\nu) >0$. Then
\begin{align}\label{dodooo}
 &\int x^{m+1}(- x^2  \lambda^2(1-q)^2;q^2)_\infty \left(q^{-m-2}\lambda^2 + A_{m,\nu} x^{-2}\right) J_\nu^{(1)}( \lambda x |q^2) d_qx =\nonumber \\&  x^{m+1}(- x^2 q^{-2} \lambda^2(1-q)^2;q^2)_\infty \left(\frac{[m]_q  J_\nu^{(1)}( \frac{\lambda x}{q} |q^2)}{q^{m}x}-q^{-m-1} D_{q^{-1}}J_\nu^{(1)}( \lambda x |q^2) \right),
\end{align}
or equivalently,
\begin{align}\label{jojok}
   &\int x^{m+1}(- x^2  \lambda^2(1-q)^2;q^2)_\infty \left(\frac{\lambda^2}{q^2} +\frac{A_{m,\nu}}{q^m x^2}\right) J_\nu^{(1)}( \lambda x |q^2) d_qx = \nonumber \\& x^{m+1}(- x^2 q^{-2} \lambda^2(1-q)^2;q^2)_\infty   \left(\frac{ [m]_q-[\nu]_q  }{x}J_\nu^{(1)}( \frac{\lambda x}{q} |q^2)+q^{\nu-1}\lambda J_{\nu+1}^{(1)}( \frac{\lambda x}{q} |q^2) \right).
\end{align}
In particular,
\begin{align*}
  &\int x^{\nu+1}(- x^2  \lambda^2(1-q)^2;q^2)_\infty  J_\nu^{(1)}( \lambda x |q^2) d_qx \\&= \frac{(qx)^{\nu+1}}{\lambda}\left(\frac{- x^2  \lambda^2}{q^2}(1-q)^2;q^2\right)_\infty J_{\nu+1}^{(1)}( \frac{\lambda x}{q} |q^2),
\end{align*}
and
\begin{align*}
&\int x^{1-\nu}(- x^2  \lambda^2(1-q)^2;q^2)_\infty  J_\nu^{(1)}( \lambda x |q^2) d_qx \\&=  - \frac{(qx)^{1-\nu}}{\lambda}\left(\frac{- x^2 \lambda^2}{q^2}(1-q)^2;q^2\right)_\infty   J_{\nu-1}^{(1)}( \frac{\lambda x}{q} |q^2),
\end{align*}
where $A_{m,\nu}$ is defined as in Theorem {$\ref{thmnbs}$}.
\end{theorem}
\begin{example}
For $\nu=0$ and $ m = 1$,  Equation (${\ref{jojok}}$) will be
\begin{align*}
  &\int (- x^2  \lambda^2(1-q)^2;q^2)_\infty (x^2 +\frac{q^2}{\lambda^2})J_0^{(1)}( \lambda x |q^2) d_qx\\& =  x^{2}(- x^2  \lambda^2(1-q)^2;q^2)_\infty \left(\frac{ q^2 J_0^{(1)}( \frac{\lambda x}{q} |q^2)}{ \lambda^2 x}+ \frac{q}{\lambda} J_1^{(1)}( \frac{\lambda x}{q} |q^2) \right).
\end{align*}
Moreover, if $\lambda $  is a  zero of $J_0^{(1)}( x|q^2)$, then
\begin{align*}
  &\int_{0}^{1} (-q^2 x^2  \lambda^2(1-q)^2;q^2)_\infty (x^2 +\frac{1}{\lambda^2})J_0^{(1)}( q\lambda x |q^2) d_qx \\&= \frac{ (- q^2  \lambda^2(1-q)^2;q^2)_\infty}{\lambda} J_1^{(1)}( \lambda |q^2).
\end{align*}
\end{example}
\begin{theorem}\label{nbhdsnjjnsb}
Let $\nu$ and $n$  be a complex number with $\Re  (\nu) > -1 $.  Let
 \begin{align*}
 & K_{n,\nu} := \Big(q^{-\nu-n-\frac{3}{2}}[\nu+n]_q-q^{\frac{-1}{2}}[\nu-n-1]_q\Big),\\&
 P_{n,\nu}(x) :=  \frac{q^{-n}[n]_q-q^{-n}[\nu]_q}{x} \sin_q(q^{-n-\frac{3}{2}}x) +q^{-n-\frac{3}{2}}\cos_q(q^{-(n+\frac{3}{2})}x),\\&
   \tilde{P}_{n,\nu}(x):= \frac{q^{-n}[n]_q-q^{-n}[\nu]_q}{x} \cos_q(q^{-(n+\frac{3}{2})}x) -q^{-n-\frac{3}{2}}\sin_q(q^{-(n+\frac{3}{2})}x).
 \end{align*}
 Then
  \begin{align}\label{fdscsshf}
   &\int (- x^2  (1-q)^2;q^2)_\infty  \Big( A_{n,\nu}  x^{n-1} \sin_q(q^{-n-\frac{3}{2}}x)  + K_{n,\nu} x^{n} \cos_q (q^{-n-\frac{3}{2}}x) \Big)  J_\nu^{(1)}(  x |q^2) d_qx \nonumber \\&=x^{n+1}(- \frac{x^2 }{q^2}(1-q)^2;q^2)_\infty  \Big( P_{n,\nu}(x) J_\nu^{(1)}( \frac{x}{q}|q^2)+q^{\nu-n-1}  \sin_q (q^{-(n+\frac{3}{2})}x) J_{\nu+1}^{(1)}( \frac{ x}{q} |q^2)\Big),
  \end{align}
  and
  \begin{align}\label{fdscshnjsr}
   &\int (- x^2 (1-q)^2;q^2)_\infty  \Big( A_{n,\nu}     x^{n-1} \cos_q(q^{-n-\frac{3}{2}}x)  - K_{n,\nu} x^{n} \sin_q(q^{-n-\frac{3}{2}}x) \Big)  J_\nu^{(1)}(  x |q^2) d_qx \nonumber \\&=x^{n+1}(-\frac{x^2 }{q^2}(1-q)^2;q^2)_\infty \Big( \tilde{P}_{n,\nu}(x) J_\nu^{(1)}( \frac{x}{q} |q^2)+ q^{\nu-n-1} \cos_q (q^{-(n+\frac{3}{2})}x) J_{\nu+1}^{(1)}( \frac{ x}{q} |q^2)\Big),
  \end{align}
  where $A_{n,\nu}$ is defined as in Theorem {$\ref{thmnbs}$}.
\end{theorem}
Corollaries {$\ref{kmjdngggdffd2511}$} and {$\ref{kmjnyyytryyr}$} follows from Theorem {$\ref{nbhdsnjjnsb}$} when $n=0$ and $n=\nu$, respectively.
\begin{corollary}\label{kmjdngggdffd2511}
For $\Re  (\nu) > -1 $,  Equations $({\ref{fdscsshf}})$ and  $({\ref{fdscshnjsr}})$, will be
\begin{align*}
  &\int (- x^2  (1-q)^2;q^2)_\infty  \Big( A_{0,\nu}  x^{-1} \sin_q(q^{-\frac{3}{2}}x)  + K_{0,\nu} \cos_q (q^{-\frac{3}{2}}x) \Big)  J_\nu^{(1)}(  x |q^2) d_qx \nonumber \\&=x(- \frac{x^2 }{q^2}(1-q)^2;q^2)_\infty  \Big( P_{0,\nu}(x) J_\nu^{(1)}( \frac{x}{q}|q^2)+q^{\nu-1}  \sin_q (q^{-\frac{3}{2}}x) J_{\nu+1}^{(1)}( \frac{ x}{q} |q^2)\Big),
  \end{align*}
and
\begin{align*}
  &\int (- x^2 (1-q)^2;q^2)_\infty  \Big( A_{0,\nu}     x^{-1} \cos_q(q^{-\frac{3}{2}}x)  - K_{0,\nu}  \sin_q(q^{-\frac{3}{2}}x) \Big)  J_\nu^{(1)}(  x |q^2) d_qx \nonumber \\&=x(-\frac{x^2 }{q^2}(1-q)^2;q^2)_\infty \Big( \tilde{P}_{0,\nu}(x) J_\nu^{(1)}( \frac{x}{q} |q^2)+ q^{\nu-1} \cos_q (q^{-\frac{3}{2}}x) J_{\nu+1}^{(1)}( \frac{ x}{q} |q^2)\Big),
  \end{align*}
respectively,
where  $K_{0,\nu}(x)$, $P_{0,\nu}(x)$ and $\tilde{P}_{0,\nu}(x) $ are defined as in Theorem {\ref{thmjdf}} with $n=0$ and
$ A_{0,\nu}$ is defined as in Theorem {$\ref{thmnbs}$} with $n=0$.
\end{corollary}
\begin{corollary}\label{kmjnyyytryyr}
For $\Re  (\nu) > -1 $,  Equations $({\ref{fdscsshf}})$and  $({\ref{fdscshnjsr}})$, will be
\begin{align*}
&\int x^\nu(- x^2  (1-q)^2;q^2)_\infty      \cos_q (q^{-(\nu+\frac{3}{2})}x)   J_\nu^{(1)}(  x |q^2) d_qx \\&=\frac{x^{\nu+1}(- x^2 q^{-2}(1-q)^2;q^2)_\infty }{q^{-\nu}[2\nu+1]_q}\Big(   \cos_q(q^{-(\nu+\frac{3}{2})}x) J_\nu^{(1)}( \frac{x}{q} |q^2)+ q^{\nu+\frac{1}{2}} \sin_q(q^{-(\nu+\frac{3}{2})}x) J_{\nu+1}^{(1)}( \frac{ x}{q} |q^2)\Big),
\end{align*}
and
\begin{align*}
&\int x^\nu(- x^2 (1-q)^2;q^2)_\infty      \sin_q(q^{-(\nu+\frac{3}{2})}x)  J_\nu^{(1)}(  x |q^2) d_qx \\&=\frac{x^{\nu+1}(- x^2 q^{-2}(1-q)^2;q^2)_\infty }{q^{-\nu} [2\nu+1]_q}\Big(  \sin_q(q^{-(\nu+\frac{3}{2})}x) J_\nu^{(1)}(\frac{ x}{q} |q^2)- q^{\nu+\frac{1}{2}} \cos_q(q^{-(\nu+\frac{3}{2})}x) J_{\nu+1}^{(1)}( \frac{ x}{q} |q^2)\Big),
\end{align*}
respectively.
\end{corollary}
\vskip0.5cm
The general solution of the  second-order $q$-difference equation
\begin{equation}\label{gg}
 \frac{1}{q}  D_{q^{-1}} D_q y(x) =   x y(x),
\end{equation}
is
  \begin{equation*}
 y(x):=  a_0  AI(x;q) + a_1  BI(x;q),
\end{equation*}
where
\begin{align*}
  AI(x;q)=\sum_{m=0}^{\infty} \dfrac{q^{\frac{m(3m+1)}{2}} (1-q)^{2m} x^{3m}}{(q^3,q^2;q^3)_m},BI(x;q) = \sum_{m=0}^{\infty} \dfrac{q^{\frac{3m(m+1)}{2}} (1-q)^{2m} x^{3m+1}}{(q^3,q^4;q^3)_m},x\in \mathbb C .
\end{align*}
\begin{theorem}\label{lap}
  For $n \in \mathbb{N}$,
  \begin{align*}
 &\int x^{n+1} \left( 1-\frac{q^{1-n}[n]_q [n-1]_q}{x^3}    \right) AI(x;q) d_qx = q^{-n} x^{n-1} \left( x D_{q^{-1}} AI(x;q)- q [n]_q  AI(\frac{x}{q};q)\right),
 \end{align*}
 and
 \begin{align*}
 &\int x^{n+1} \left( 1-\frac{q^{1-n}[n]_q [n-1]_q}{x^3}    \right)  BI(x;q) d_qx = q^{-n} x^{n-1} \left( x D_{q^{-1}}  BI(x;q)- q [n]_q   BI(\frac{x}{q};q)\right) .
 \end{align*}
 \end{theorem}
\begin{proof}
By comparing  Equation (${\ref{gg}}$) with Equation (${\ref{mxlk}}$), we get \[ p(x)=0 ,\quad   r(x) = - x.\]  Then the solution of Equation ({\ref{fht}}) is $f(x) =1 $.
Substituting with $h(x) = x^n$ into Equation (${\ref{mkl}}$), we get the desired result.
\end{proof}
\begin{corollary}
 \begin{align*}
 \int x AI(x;q) d_qx =D_{q^{-1}} AI(x;q),
 \end{align*}
 \begin{align*}
 \int x BI(x;q) d_qx =D_{q^{-1}} BI(x;q),
 \end{align*}
   \begin{align*}
  \int x^2 AI(x;q) d_qx = - AI(\frac{x}{q};q)+\frac{x}{q} D_{q^{-1}} AI(x;q),
 \end{align*}
 and
 \begin{align*}
  \int x^2 BI(x;q) d_qx = - BI(\frac{x}{q};q)+\frac{x}{q} D_{q^{-1}} BI(x;q).
 \end{align*}
 \end{corollary}
  \begin{proof}
  The proof of the first two $q$-integrals and the second two $q$-integrals follow by substituting with $n=0$ and $n=1$ in Theorem {\ref{lap}}, respectively.
 \end{proof}
 \begin{theorem}
 We have
\begin{align*}
 & \int \left(1+x\right) \sin(x;q)  AI(x;q) d_qx =  \sin(\frac{x}{q};q)  D_{q^{-1}}AI(x;q) - \cos(q^{-\frac{1}{2}}x;q) AI(q^{-1}x;q),
\end{align*}
and
\begin{align*}
 & \int \left(1+x\right) \sin(x;q)  BI(x;q) d_qx =  \sin(\frac{x}{q};q)  D_{q^{-1}}BI(x;q) - \cos(q^{-\frac{1}{2}}x;q) BI(q^{-1}x;q).
\end{align*}
\end{theorem}
 \begin{proof}
   The proof follows immediately by substituting with $h(x) = \sin(x;q)$ into Equation (${\ref{mkl}}$).
 \end{proof}
\begin{theorem}
Let $a,b,c \in \mathbb{C}$, $c \neq 0$,  $\alpha = \frac{[a]_q [b]_q}{[c]_q}$, $\beta=\frac{[a+b+1]_q}{[c]_q}$, $\alpha_1=\beta-\alpha $ and
\begin{align*}
  &T_\alpha(x) :=\frac{\alpha}{(1+\alpha (1-q) x )} -\frac{[a+b+1]_q}{q^{c+1} (1-q^{a+b-c}x)},\\&
 S_{\alpha_1}(x):= \frac{q \alpha_1 -\beta}{(1-\alpha_1 x)(q-\beta x)}.
\end{align*}
  Then
\begin{align}\label{2145hgf}
  &\int x^{c-1} (x;q)_{a+b-c} \,_2\phi_1(q^a,q^b;q^c;q,x) d_qx =  \frac{x^c }{[c]_q}  (\frac{x}{q};q)_{a+b+1-c} \,_2\phi_1(q^{a+1},q^{b+1};q^{c+1};q,\frac{x}{q}),
\end{align}
\begin{align}\label{2146hgf}
  &\int x^{c-1} (x;q)_{a+b-c} \left([c]_q-\frac{x}{q} [a+1]_q [b+1]_q \right)\,_2\phi_1(q^a,q^b;q^c;q,x) d_qx =\nonumber \\& x^c (\frac{x}{q};q)_{a+b+1-c}  \left(  \,_2\phi_1(q^a,q^b;q^c;q,x)-\frac{[a]_q [b]_q x}{ [c]_q} \,_2\phi_1(q^{a+1},q^{b+1};q^{c+1};q,\frac{x}{q})\right),
\end{align}
\begin{align}\label{2147hgf}
  &\int x^{c} (x;q)_{a+b+1-c} (- \alpha(1-q)x  ;q)_\infty T_\alpha(x) \,_2\phi_1(q^a,q^b;q^c;q,x) d_qx =\nonumber \\& \frac{x^c}{q^c} (\frac{x}{q};q)_{a+b+1-c}(-\alpha (1-q) x ;q)_\infty \left(\,_2\phi_1(q^a,q^b;q^c;q,x)- \,_2\phi_1(q^{a+1},q^{b+1};q^{c+1};q,\frac{x}{q})\right),
\end{align}
and
\begin{align}\label{2148hgf}
 &\int \frac{x^c (\alpha_1 x;q)_\infty(x;q)_{a+b+1-c}}{(\beta x;q)_\infty}S_{\alpha_1}(x) \,_2\phi_1(q^a,q^b;q^c;q,x) d_qx =  - q^{-c}(1-q)x^c\nonumber \\& \frac{ (\alpha_1 x;q)_\infty (\frac{x}{q};q)_{a+b+1-c}}{(q^{-1} \beta x;q)_\infty} \left( \,_2\phi_1(q^a,q^b;q^c;q,x)-(1-q)(1-\frac{\beta}{q}x)\,_2\phi_1(q^{a+1},q^{b+1};q^{c+1};q,\frac{x}{q}) \right).
\end{align}
\end{theorem}
\begin{proof}
The $q$-hypergeometric functions $\,_2\phi_1(q^a,q^b;q^c;q,x)$  satisfies the second-order $q$-difference equation \cite{rahman}
\begin{align}\label{kjhgfd}
&\frac{1}{q}  D_{q^{-1}} D_q y(x) + \dfrac{[c]_q - [a+b+1]_q \frac{x}{q}}{ x  (q^c - q^{a+b}x)} D_{q^{-1}} y(x)- \dfrac{[a]_q [b]_q}{x (q^c-q^{a+b}x)} y(x) =0.
\end{align}
 By comparing  Equation (${\ref{kjhgfd}}$) with Equation (${\ref{mxlk}}$),  we get
\begin{align*}
 p(x)= \dfrac{(1-q^c) - (1-q^{a+b+1}) \frac{x}{q}}{ x (1-q) (q^c - q^{a+b}x)} , \quad r(x) = - \dfrac{(1-q^a)(1-q^b)}{x(1-q)^2 (q^c-q^{a+b}x)}.
\end{align*}
If $f(x)$ satisfies Equation (${\ref{fht}}$), then
\begin{align*}
 f(x)  = x^c (x;q)_{a+b+1-c}.
\end{align*}
The proof of Equation (${\ref{2145hgf}}$), follows by
substituting  with $h (x) =1$ into Equation (${\ref{mkl}}$), and using
\begin{align}\label{32514}
  D_{q^{-1}} \,_2\phi_1(q^a,q^b;q^c;q,x) = \frac{(1-q^a)(1-q^b)}{(1-q^c)(1-q)}\,_2\phi_1(q^{a+1},q^{b+1};q^{c+1};q,\frac{x}{q}).
\end{align}
 Equation (${\ref{2146hgf}}$) follows by
substituting with $h (x) =x$ into Equation (${\ref{mkl}}$).
 Equation (${\ref{2147hgf}}$) follows by
taking  $h (x)$  to be a  solution of
\begin{align*}
  \dfrac{(1-q^c)}{ x (1-q) (q^c - q^{a+b}x)} D_{q^{-1}} h(x)- \dfrac{(1-q^a)(1-q^b)}{x(1-q)^2 (q^c-q^{a+b}x)} h(x) =0,
\end{align*}
I.e.
\begin{align*}
h(x) = (- \alpha(1-q)x ;q)_\infty, \quad\quad \alpha = \frac{[a]_q [b]_q}{[c]_q}.
\end{align*}
 Finally, Equation (${\ref{2148hgf}}$) follows by
taking $h (x)$ to be a solution of
\begin{align*}
  \dfrac{(1-q^c) - (1-q^{a+b+1}) \frac{x}{q}}{ x (1-q) (q^c - q^{a+b}x)} D_{q^{-1}} h(x)- \dfrac{(1-q^a)(1-q^b)}{x(1-q)^2 (q^c-q^{a+b}x)}h(x) =0.
\end{align*}
Thus,
\begin{align*}
h(x) = \dfrac{(\alpha_1 x;q)_\infty}{(\beta x;q)_\infty}\quad \quad \beta=\frac{[a+b+1]_q}{[c]_q}, \quad \alpha_1=\beta-\alpha.
\end{align*}
\end{proof}
\begin{theorem}\label{55}
  If $ p_n(x;a,b;q)$ is the big $q$-Laguerre  polynomial of degree $n$, then
  \begin{align}\label{254}
  &\int \dfrac{(\frac{x}{b},\frac{x}{a};q)_\infty }{(x ;q)_\infty}p_n(x;a,b;q) d_qx = \dfrac{abq^2 (1-q) (\frac{x}{bq},\frac{x}{aq};q)_\infty }{(1-aq)(1-bq)(x ;q)_\infty} p_{n-1}(x;aq,bq;q),
\end{align}
and
\begin{align}\label{desf2gd}
  & \int \dfrac{x^n(\frac{x}{b},\frac{x}{a};q)_\infty }{(qx ;q)_\infty}\left(\frac{[n-1]_q}{x^2}-\frac{(a+b-qab)}{abqx(1-q)(1-x)}\right)p_n(x;a,b;q) d_qx =\nonumber\\&
 q^{-n} x^n \dfrac{(\frac{x}{bq},\frac{x}{aq};q)_\infty }{(x ;q)_\infty}\left(\frac{q^n p_n(\frac{x}{q};a,b;q)}{x}-\frac{p_{n-1}(x;aq,bq;q)}{(1-aq)(1-bq)}\right).
\end{align}
\end{theorem}
\begin{proof}
The big $q$-Laguerre  polynomial $ p_n(x;a,b;q)= \,_3\phi_2\left( \begin{array}{cccc}
                     q^{-n} , 0,x \\
                    aq,bq
                   \end{array}
\mid q;q\right) $  satisfies the second-order $q$-difference equation, see \cite[Eq.(3.11.5)]{askey},
\begin{align}\label{kjhgf}
    \frac{1}{q} D_{q^{-1}} D_qy(x)+  \dfrac{ x-q(a+b-qab)}{ab q^2  (1-q)(1-x)}D_{q^{-1}}y(x) - \frac{ q^{-n-1} [n]_q}{ab(1-q)(1-x)} y(x) =0.
\end{align}
  By comparing  Equation (${\ref{kjhgf}}$) with Equation (${\ref{mxlk}}$), we get
  \begin{align*}
    p(x)= \dfrac{ x-q(a+b-qab)}{ab q^2  (1-q)(1-x)}\quad \quad r(x) =  - \frac{ q^{-n-1} [n]_q}{ab(1-q)(1-x)} .
  \end{align*}
Thus $f(x) =  \frac{(\frac{x}{b},\frac{x}{a};q)_\infty }{(qx;q)_\infty}$. The proof of (${\ref{254}}$) follows by
substituting with $h (x) =\frac{abq^{n+1}(1-q)}{[n]_q}$ into Equation (${\ref{mkl}}$) and using  \cite[Eq.(3.11.7)]{askey} (with $x$ is replaced by $\frac{x}{q}$)
\begin{align}\label{246}
  D_{q^{-1}}p_n(x;a,b;q)=\dfrac{q^{1-n}[n]_q}{(1-aq)(1-bq)}p_{n-1}(x;aq,bq;q).
\end{align}
The proof of (${\ref{desf2gd}}$) follows by taking
 $h(x)$ as the solution of
\begin{align*}
  \dfrac{x}{ab q^2  (1-q)(1-x)}D_{q^{-1}}h(x) - \frac{ q^{-n-1} [n]_q}{ab(1-q)(1-x)} h(x) =0,
\end{align*}
which gives $ h(x) = x^n$.
\end{proof}
\begin{example}

 If $a=1$ in Equation (${\ref{254}}$), then we get
  \begin{align*}
   &\int (\frac{x}{b};q)_\infty p_n(x;1,b;q) d_qx =  \dfrac{bq^2}{ 1-bq}(1-\frac{x}{q})(\frac{x}{bq};q)_\infty p_{n-1}(x;q,bq;q).
\end{align*}
\end{example}
\begin{remark}
  Equation (${\ref{254}}$) is equivalent to \cite[Eq.(3.11.9)]{askey} (with $n$ is replaced by $n-1$ )
\begin{align*}
  D_q\left( w(x;aq,bq;q)p_{n-1}(x;aq,bq;q) \right) = \frac{(1-aq)(1-bq)}{abq^2(1-q)}w(x;a,b;q)p_{n}(x;a,b;q),
\end{align*}
where $ w(x;q)= \frac{(\frac{x}{b},\frac{x}{a};q)_\infty }{(x;q)_\infty}$ .
\end{remark}
\begin{proposition}
If $n$ and $m$ are non-negative integers, then
  \begin{align*}
    &\int \dfrac{(\frac{x}{b},\frac{x}{a};q)_\infty }{(x ;q)_\infty}\left(q^{-m}[m]_q-q^{-n}[n]_q\right)p_m(x;a,b;q)p_n(x;a,b;q) d_qx = \dfrac{abq^2 (1-q) (\frac{x}{bq},\frac{x}{aq};q)_\infty }{(1-aq)(1-bq)(x ;q)_\infty}\\&\left(q^{-m} [m]_q p_{m-1}(x;aq,bq;q)p_n(\frac{x}{q};a,b;q)-q^{-n} [n]_q p_m(\frac{x}{q};a,b;q) p_{n-1}(x;aq,bq;q)\right).
  \end{align*}
\end{proposition}
\begin{proof}
  The proof follows by sbstituting with  $h(x)=p_m(x;a,b;q)$ and $y(x) =p_n(x;a,b;q) $ in (${\ref{mkl}}$).
\end{proof}
\begin{theorem}\label{56}
 If $ L_n^{\alpha}(x;q)$ is the  $q$-Laguerre  polynomial of degree $n$ and $\mu = 1-\frac{\ln(1+q-q^n)}{\ln q}$,  then
 \begin{align}\label{214dsbg}
  &\int \dfrac{x^{m+\alpha} }{(-x;q)_\infty}\left(\frac{(1-q^{m+\alpha})[m]_q}{x}+q^{m+\alpha}([n]_q-[m]_q)\right)L_n^{\alpha}(x;q) d_qx \nonumber \\&= \dfrac{ x^{m+\alpha+1}}{(-x;q)_\infty}\left( \frac{(1-q^m)}{x}L_n^{\alpha}(\frac{x}{q};q)+ q^{\alpha} L_{n-1}^{\alpha+1}(x;q)\right),
\end{align}
\begin{align}\label{5548}
&\int \frac{x^{\alpha} (q^{\alpha+2}(1-q^n)x;q)_\infty}{(-qx;q)_\infty}  L_n^{\alpha}(x;q) d_qx=\nonumber\\&\frac{x^{\alpha+1} (q^{\alpha+1}(1-q^n)x;q)_\infty}{q^{\alpha+1}(-x;q)_\infty} \left(\frac{1-q^{\alpha}(1-q^n)x}{[n]_q}L_{n-1}^{\alpha+1}(x;q)  - (1-q)L_n^{\alpha}(\frac{x}{q};q)\right),
\end{align}
and
\begin{align}\label{jnhbvfddds}
 &\int \dfrac{x^{\alpha+\mu-1} }{(-qx;q)_\infty}\left(q[\mu-1]_q+\frac{[\alpha+1]_q}{q^{\alpha}(1+x)}\right)L_n^{\alpha}(x;q) d_qx \nonumber \\&= \dfrac{ x^{\alpha+\mu+1}}{(-x;q)_\infty}\left( \frac{L_n^{\alpha}(\frac{x}{q};q)}{q^{\alpha}x}-  \frac{L_{n-1}^{\alpha+1}(x;q)}{1-q^\mu}\right).
\end{align}
\end{theorem}
\begin{proof}
The  $q$-Laguerre  polynomial
\begin{align*}
  L_n^{\alpha}(x;q)= \frac{1}{(q;q)_n}\,_2\phi_1\left( \begin{array}{cccc}
                     q^{-n} ,-x \\
                    0
                   \end{array}
\mid q;q^{n+\alpha+1}\right),  \quad  \quad  \alpha > -1
\end{align*}
   satisfies the second-order $q$-difference equation, see \cite[Eq.(3.21.6)]{askey},
\begin{align}\label{kjdhg}
    \frac{1}{q} D_{q^{-1}} D_qy(x)+  \dfrac{ 1-q^{\alpha+1}(1+x)}{q^{\alpha+1}x(1+x)(1-q)}D_{q^{-1}}y(x) + \frac{ [n]_q}{x(1-q)(1+x)} y(x) =0.
\end{align}

  By comparing  Equation (${\ref{kjdhg}}$) with Equation (${\ref{mxlk}}$), we obtain
  \begin{align*}
    p(x)= \dfrac{ 1-q^{\alpha+1}(1+x)}{q^{\alpha+1}x(1+x)(1-q)},\quad \quad r(x) = \frac{ [n]_q}{x(1-q)(1+x)}.
  \end{align*}
Hence $ f(x) =  \dfrac{x^{\alpha+1} }{(-qx;q)_\infty}$
is a solution of Equation ({\ref{fht}}). The proof of (${\ref{214dsbg}}$) follows by
substituting  with $h (x) =x^m$ into Equation (${\ref{mkl}}$), and using  \cite[Eq.(3.21.8)]{askey} (with $x$ is replaced by $\frac{x}{q}$ )
\begin{align}\label{4587}
  D_{q^{-1}}L_n^{\alpha}(x;q)=\dfrac{- q^{\alpha+1} }{(1-q)}L_{n-1}^{\alpha+1}(x;q).
\end{align}
Equation (${\ref{5548}}$) follows by taking $h(x)$ as the  solution of
\begin{align*}
  \dfrac{ 1}{q^{\alpha+1}(1-q) x(1+x)}D_{q^{-1}}h(x) + \frac{ [n]_q}{x(1-q)(1+x)} h(x) =0,
\end{align*}
which gives $ h(x) = (q^{\alpha+1}(1-q^n)x;q)_\infty$.
Equation  (${\ref{jnhbvfddds}}$) follows by taking $h(x)$ as the  solution of
\begin{align*}
  \dfrac{ -1}{(1-q)(1+x)}D_{q^{-1}}h(x) + \frac{ [n]_q}{x(1-q)(1+x)} h(x) =0,
\end{align*}
which gives $ h(x) = x^{\mu}$, $\mu = 1-\frac{\ln(1+q-q^n)}{\ln q}$.
\end{proof}
\begin{remark}
  It is worth noting that
   \begin{align*}
  &\int \dfrac{x^\alpha }{(-x;q)_\infty}L_n^{\alpha}(x;q) d_qx = \dfrac{x^{\alpha+1}   }{[n]_q(-x;q)_\infty} L_{n-1}^{\alpha+1}(x;q),
\end{align*}
  which is the case $m=0$ in (${\ref{214dsbg}}$) is equivalent to \cite[Eq.(3.21.10)]{askey} (with $\alpha$ is replaced by $\alpha+1$ and $n$ is replaced by $n-1$ )
\begin{align*}
  D_q\left( w(x;\alpha+1;q) L_{n-1}^{\alpha+1}(x;q) \right) = [n]_q w(x;\alpha;q)L_{n}^{\alpha}(x;q),
\end{align*}
where $ w(x;\alpha;q)= \displaystyle{\frac{x^\alpha}{(-x;q)_\infty}}$.
\end{remark}
\begin{remark}
  If $h(x)$ is the solution of
\begin{align*}
  \dfrac{ -1}{x(1-q)}D_{q^{-1}}h(x) + \frac{ (1-q^n)}{x(1-q)^2(1+x)} h(x) =0,
\end{align*}
then $h(x) = \dfrac{(-x(1+q-q^n);q)_\infty}{(-qx;q)_\infty}$.
If $n=1$, then $h(x)=1+x $, and
\begin{align*}
  &\int \dfrac{x^\alpha }{(-x;q)_\infty} \left(1-q^{\alpha+1}(1+x)\right)d_qx = \frac{x^{\alpha+1}}{(-x;q)_\infty}.
\end{align*}
\end{remark}
\begin{proposition}
  If $m$ and $n$ are non-negative integers, then
  \begin{align*}
  \left([n]_q-[m]_q\right) &\int \dfrac{x^{\alpha} }{(-x;q)_\infty}L_n^{\alpha}(x;q)L_m^{\alpha}(x;q) d_qx \nonumber \\&= \dfrac{ x^{\alpha+1}}{(-x;q)_\infty}\left( L_m^{\alpha}(\frac{x}{q};q)L_{n-1}^{\alpha+1}(x;q)-L_{m-1}^{\alpha+1}(x;q) L_n^{\alpha}(\frac{x}{q};q)\right).
  \end{align*}
\end{proposition}
\begin{proof}
  The proof follows bysubstituting with  $h(x)=L_m^{\alpha}(x;q)$ and $y(x) = L_n^{\alpha}(x;q) $ in (${\ref{mkl}}$).
\end{proof}
\vskip0.5cm
The Stieltjes and Hamburger moment problem associated with the Stieltjes-Wigert polynomials is indeterminate, and the polynomials are orthogonal to many different weight functions. For example, they are orthogonal to the weight function
\begin{align*}
 w(x)=\frac{-c}{\sqrt{\pi}} e^{c\ln^2 x}\quad  x > 0 \quad and \quad  w(x)=\frac{x^2}{(-x,\frac{-q}{x};q)_\infty},
\end{align*}
see \cite{chrit,askey}.
\begin{theorem}\label{57}
  If $S_n(x;q)$ is  the  Stieltjes-Wigert  polynomial of degree $n$, $c=\frac{1}{2\ln q}$, $\alpha_n= 1-\frac{\ln(1-q^n+q)}{\ln q}$, then
  \begin{align}\label{sum3}
&\int x^{\frac{-1}{2}} e^{c\ln^2 x} S_n(x;q) d_qx=\frac{x^{\frac{3}{2}}e^{c\ln^2 \frac{x}{q}}}{\sqrt{q} [n]_q} S_{n-1}(qx;q),
\end{align}
\begin{align}\label{214gg}
 &\int x^{m-\frac{1}{2}} e^{c\ln^2 x} \left(q^{2m} [n-m]_q+\frac{[m]_q}{x}\right)S_n(x;q) d_qx \nonumber\\&=q^{\frac{-1}{2}}
  x^{m+\frac{3}{2}}  e^{c\ln^2 \frac{x}{q}}\left(\frac{(1-q^m)S_n(\frac{x}{q};q)}{x}+S_{n-1}(qx;q)\right),
\end{align}
\begin{align}\label{lak}
&\int x^{\frac{-1}{2}} e^{c\ln^2 x} (q^2x;q)_\infty\left(-q^{n-1}+(1+q^n)x\right) S_n(x;q) d_qx=\nonumber\\&q^{\frac{-3}{2}}
x^{\frac{3}{2}}  e^{c\ln^2 \frac{x}{q}}(qx;q)_\infty\left((1-q)(1-x)S_{n-1}(qx;q)-(1-q)S_n(\frac{x}{q};q)\right),
\end{align}
\begin{align}\label{214fgdg}
  &\int x^{-\frac{1}{2}} e^{c\ln^2 x} (-q^{n+2}x;q)_\infty  S_n(x;q) d_qx=\nonumber\\&q^{\frac{-1}{2}}(1-q)
x^{\frac{3}{2}}  e^{c\ln^2 \frac{x}{q}}(-q^{n+1}x;q)_\infty\left((1+q^nx)S_{n-1}(qx;q)+q^nS_n(\frac{x}{q};q)\right),
\end{align}
and
\begin{align}\label{sum2}
  &\int x^{\alpha_n-\frac{1}{2}} e^{c\ln^2 x} \left([\alpha_n-1]_q+\frac{1}{qx(1-q)}\right)S_n(x;q) d_qx\nonumber\\&=q^{\frac{-3}{2}}
  x^{\alpha_n+\frac{3}{2}}  e^{c\ln^2 \frac{x}{q}}\left(\frac{S_n(\frac{x}{q};q)}{x}+\frac{S_{n-1}(qx;q)}{(1-q^{\alpha_n})}\right).
\end{align}

\end{theorem}
\begin{proof}
  The  Stieltjes-Wigert  polynomials $S_n(x;q)=\frac{1}{(q;q)_n}\,_1\phi_1\left( \begin{array}{cccc}
                     q^{-n}  \\
                    0
                   \end{array}
\mid q;-q^{n+1}x\right) $  satisfies the second-order $q$-difference equation, see \cite[Eq.(3.27.5)]{askey},
\begin{align}\label{kjdshg}
    \frac{1}{q} D_{q^{-1}} D_qy(x)+  \dfrac{ 1-qx}{q x^2 (1-q)}D_{q^{-1}}y(x) + \frac{ [n]_q}{x^2(1-q)} y(x) =0.
\end{align}
By comparing  Equation (${\ref{kjdshg}}$) with Equation (${\ref{mxlk}}$), we get
  \begin{align*}
    p(x)= \dfrac{ 1-qx}{q x^2 (1-q)},\quad \quad \quad r(x) =  \frac{ [n]_q}{x^2(1-q)} .
  \end{align*}
  Thus $f(x) = x^{\frac{3}{2}} e^{c \ln^2x}$,  $(c=\frac{1}{2\ln q})$  is a solution of Equation (${\ref{fht}}$). The proof of (${\ref{sum3}}$) follows by
substituting  with $h (x) =\frac{1-q}{[n]_q}$ into Equation (${\ref{mkl}}$), and using \cite[Eq.(3.27.7)]{askey} (with $x$ is replaced by $\frac{x}{q}$ )
\begin{align}\label{4666vf}
   D_{q^{-1}}S_n(x;q)=\frac{-q}{1-q}S_{n-1}(qx;q).
\end{align}
Substituting  with $h (x) =x^m$ into Equation (${\ref{mkl}}$), and using  (${\ref{4666vf}}$) yields (${\ref{214gg}}$). The  proof of (${\ref{lak}}$) follows by
taking $h(x)$ as a solution of
\begin{align*}
  \dfrac{ 1}{qx^2(1-q)}D_{q^{-1}}h(x) + \frac{ 1}{x^2(1-q)^2} h(x) =0.
\end{align*}
I.e. $ h(x) = (qx;q)_\infty$.
Equation (${\ref{214fgdg}}$) follows by
taking $h(x)$ as a solution of
\begin{align*}
  \dfrac{ 1}{qx^2(1-q)}D_{q^{-1}}h(x) - \frac{q^n }{x^2(1-q)^2} h(x) =0,
\end{align*}
which give $ h(x) = (-q^{n+1}x;q)_\infty $.
Equation (${\ref{sum2}}$) follows by
taking $h(x)$ as a solution of
\begin{align*}
  \dfrac{ -1}{x(1-q)}D_{q^{-1}}h(x) + \frac{ [n]_q}{x^2(1-q)} h(x) =0.
\end{align*}
Hence,  $ h(x) = x^{\alpha_n},\quad  \alpha_n = 1-\dfrac{\ln(1+q-q^n)}{\ln q}$.
\end{proof}
The following result follows from Theorem ${\ref{57}}$, by calculating the indefinite $q$-integral from 0 to $a$.
\begin{corollary}
For $a>0$,
  \begin{align*}
    & \sum_{k=0}^{\infty} q^{\frac{(k-1)^2}{2}} a^k  S_n(q^k a;q) =  \frac{S_{n-1}(qa;q)}{(1-q^n)}.
  \end{align*}\end{corollary}
  \begin{proof}
    The proof follows by Using $({\ref{nun}})$  in   $({\ref{sum3}})$.
  \end{proof}
  \begin{proposition}
Let $n$ and $m$ be non-negative integers.  Then
  \begin{align*}
  \left([n]_q-[m]_q\right) &\int x^{-\frac{1}{2}} e^{c\ln^2 x} S_m(x;q) S_n(x;q) d_qx \nonumber\\&= q^{-\frac{1}{2}} x^{\frac{3}{2}}e^{c\ln^2 \frac{x}{q}}  \left(S_m(\frac{x}{q};q)S_{n-1}(qx;q)-S_{m-1}(qx;q)S_n(\frac{x}{q};q)\right).
  \end{align*}
  \end{proposition}
  \begin{proof}
    The proof follows by substituting with  $h(x)=S_m(x;q)$ and $y(x) = S_n(x;q) $ in (${\ref{mkl}}$).
  \end{proof}
\begin{example}~\par
\vskip .5 cm

\begin{itemize}
  \item If $m=n$ in (${\ref{214gg}}$), then
\begin{align*}
&\int x^{n-\frac{3}{2}} e^{c\ln^2 x} S_n(x;q) d_qx =q^{\frac{-1}{2}}
  x^{n+\frac{3}{2}}  e^{c\ln^2 \frac{x}{q}}\left(\frac{1-q}{x}S_n(\frac{x}{q};q)+\frac{S_{n-1}(qx;q)}{[n]_q}\right).
\end{align*}
  \item
  If $n=1$ in  (${\ref{lak}}$), then
\begin{align*}
&\int x^{\frac{-1}{2}} e^{c\ln^2 x} (q^2x;q)_\infty (1-(1+q)x)(1-qx) d_qx=\nonumber q^{\frac{-1}{2}}(1-q)
x^{\frac{3}{2}}  e^{c\ln^2 \frac{x}{q}}(x;q)_\infty.
\end{align*}
\end{itemize}
\end{example}
\begin{theorem}
Let $n$ and $m$ be non-negative integers. If $S_n(x;q)$ is  the  Stieltjes-Wigert  polynomial of degree $n$, then
\begin{align}\label{214dd}
 &\int \frac{x^{m}}{(-x,\frac{-q}{x};q)_\infty} \left(q^{2m} [n-m]_q+\frac{[m]_q}{x}\right)S_n(x;q) d_qx \nonumber\\&= \frac{
  x^{m+2}}{q(\frac{-x}{q},\frac{-q^2}{x};q)_\infty }  \left(\frac{(1-q^m)S_n(\frac{x}{q};q)}{x}+S_{n-1}(qx;q)\right),
\end{align}
\begin{align}\label{lak2}
&\int \frac{ (q^2x;q)_\infty}{(-x,\frac{-q}{x};q)_\infty}\left(-q^{n-1}+(1+q^n)x\right) S_n(x;q) d_qx\nonumber\\&=\frac{(1-q)x^2(qx;q)_\infty }{q^2 (\frac{-x}{q},\frac{-q^2}{x};q)_\infty}\left((1-x)S_{n-1}(qx;q)-S_n(\frac{x}{q};q)\right),
\end{align}
\begin{align}\label{hgfsbff}
   &\int \frac{ (-q^{n+2}x;q)_\infty }{(-x,\frac{-q}{x};q)_\infty} S_n(x;q) d_qx=\nonumber\\&\frac{(1-q)x^2(-q^{n+1}x;q)_\infty }{q (\frac{-x}{q},\frac{-q^2}{x};q)_\infty}\left((1+q^nx)S_{n-1}(qx;q)+q^n S_n(\frac{x}{q};q)\right),
\end{align}
and
\begin{align}\label{gbdffss}
   &\int\frac{ x^{\alpha_n}}{(-x,\frac{-q}{x};q)_\infty}  \left([\alpha_n-1]_q+\frac{1}{qx(1-q)}\right)S_n(x;q) d_qx\nonumber\\&=
  \frac{x^{\alpha_n+2}}{q^2(\frac{-x}{q},\frac{-q^2}{x};q)_\infty}  \left(\frac{ S_n(\frac{x}{q};q)}{x}-\frac{S_{n-1}(qx;q)}{(1-q^{\alpha})}\right).
\end{align}
\end{theorem}
\begin{proof}
  From Equation (${\ref{kjdshg}}$), we get $f(x)= \frac{x^2}{(-x,\frac{-q}{x};q)_\infty}$ is a
 solution of Equation (${\ref{fht}}$).
Substituting  with $h (x) =x^m$ into Equation (${\ref{mkl}}$), and using  (${\ref{4666vf}}$) yields (${\ref{214dd}}$). The  proof of ({\ref{lak2}}) follows by
taking $h(x)$ as a solution of
\begin{align*}
  \dfrac{ 1}{qx^2(1-q)}D_{q^{-1}}h(x) + \frac{ 1}{x^2(1-q)^2} h(x) =0,
\end{align*}
which give $ h(x) = (qx;q)_\infty$.
Equation (${\ref{hgfsbff}}$) follows by
taking $h(x)$ as a solution of
\begin{align*}
  \dfrac{ 1}{qx^2(1-q)}D_{q^{-1}}h(x) - \frac{q^n }{x^2(1-q)^2} h(x) =0.
\end{align*}
I.e. $ h(x) = (-q^{n+1}x;q)_\infty $.
Equation (${\ref{gbdffss}}$) follows by
taking $h(x)$ as a solution of
\begin{align*}
  \dfrac{ -1}{x(1-q)}D_{q^{-1}}h(x) + \frac{ [n]_q}{x^2(1-q)} h(x) =0.
\end{align*}
Thus $ h(x) = x^{\alpha_n},\quad  \alpha_n = 1-\dfrac{\ln(1+q-q^n)}{\ln q}$.
\end{proof}
\begin{remark}
 The case $m=0$ in (${\ref{214dd}}$) is
  \begin{align*}
   &\int \frac{ S_n(x;q)}{(-x,\frac{-q}{x};q)_\infty} d_qx=\frac{x^{2}}{q [n]_q(\frac{-x}{q},\frac{-q^2}{x};q)_\infty } S_{n-1}(qx;q),
  \end{align*}
  equivalent to  \cite[Eq.(3.27.9)]{askey} (with $x$ is replaced by $qx$ and $n$ is replaced by $n-1$ )
\begin{align*}
  D_q\left( w(qx;q) S_{n-1}(qx;q) \right) =q^{-1} [n]_q w(x;q)S_{n}(x;q),
\end{align*}
where  $ w(x;q)= \displaystyle{\frac{1}{(-x,\frac{-q}{x};q)_\infty}}$.
\end{remark}
\begin{corollary}
  Let $n$ and $m$ be non-negative integers.  Then
 \begin{align}\label{jnsbsbsbsb}
     &\int_{0}^{\infty} \frac{1}{(-x,\frac{-q}{x};q)_\infty}S_m(x;q) S_n(x;q) dx =0 \quad if \quad m\neq n,
   \end{align}
   which is consistent with  the orthogonality relation \cite[Eq.(3.27.2)]{askey} .
\end{corollary}
\begin{proof}
  The proof follows by taking  $h(x)=S_m(x;q)$,  we get
  \begin{align}\label{hbgaanha}
  \left([n]_q-[m]_q\right) &\int \frac{1}{(-x,\frac{-q}{x};q)_\infty}S_m(x;q) S_n(x;q) d_x \nonumber\\&= \frac{
  x^{2}}{q(\frac{-x}{q},\frac{-q^2}{x};q)_\infty }  \left(S_m(\frac{x}{q};q)S_{n-1}(qx;q)-S_{m-1}(qx;q)S_n(\frac{x}{q};q)\right).
  \end{align}
   Applying Theorem {$\ref{bvbvnc}$}, gives (${\ref{jnsbsbsbsb}}$).
\end{proof}
\begin{example}~\par
\vskip .5 cm

\begin{itemize}
  \item  If $n=m$ in  (${\ref{214dd}}$), then
\begin{align*}
&\int \frac{x^{n-1}}{(-x,\frac{-q}{x};q)_\infty} S_n(x;q) d_qx = \frac{
  x^{n+2}}{q(\frac{-x}{q},\frac{-q^2}{x};q)_\infty }  \left(\frac{1-q}{x}S_n(\frac{x}{q};q)+\frac{S_{n-1}(qx;q)}{[n]_q}\right).
\end{align*}
  \item If $n=1$ in  (${\ref{lak2}}$), then
\begin{align*}
&\int \frac{ (q^2x;q)_\infty}{(-x,\frac{-q}{x};q)_\infty}(1-x-qx)(1-qx)  d_qx=\nonumber\frac{1}{q}(1-q)x^2\frac{(x;q)_\infty }{ (\frac{-x}{q},\frac{-q^2}{x};q)_\infty}.
\end{align*}
\end{itemize}
\end{example}
\begin{theorem}\label{58}
 If $h_n(x;q)$ is  the  discrete $q$-Hermite I  polynomial of degree $n$, then
   \begin{align}\label{214dddfc}
 & \int x^{m} (q^2x^2;q^2)_\infty \left(\dfrac{[m]_q [m-1]_q}{x^2}-\frac{[m-n]_q}{1-q}\right)h_n(x;q) d_qx \nonumber\\&=
 x^{m}(x^2;q^2)_\infty \left(\frac{[m]_q h_n(\frac{x}{q};q)}{x}-q^{-1}[n]_q h_{n-1}(\frac{x}{q};q)  \right).
\end{align}
\end{theorem}
\begin{proof}
The discrete $q$-Hermite I polynomials is defined by  $$h_n(x;q)= q^{\binom{n}{2}} \,_2\phi_1\left( \begin{array}{cccc}
                     q^{-n},x^{-1}  \\
                    0
                   \end{array}
\mid q;-q x\right), n\in \mathbb{N}_0 $$  which satisfies the second-order $q$-difference equation, see \cite[Eq.(3.28.5)]{askey},
\begin{align}\label{kjdsmhgb}
    \frac{1}{q} D_{q^{-1}} D_qy(x)- \dfrac{x}{1-q}D_{q^{-1}}y(x) + \frac{ q^{1-n}[n]_q}{1-q} y(x) =0.
\end{align}
  By comparing  Equation (${\ref{kjdsmhgb}}$) with Equation (${\ref{mxlk}}$), we get
  \begin{align*}
    p(x)= - \dfrac{x}{1-q}, \quad \quad r(x) =   \frac{ q^{1-n}[n]_q}{1-q} .
  \end{align*}
Then we get  $f(x) =   (q^2x^2;q^2)_\infty $ is a solution of Equation (${\ref{fht}}$).
Substituting with  $h (x) =x^m$ into Equation (${\ref{mkl}}$),
 and using \cite[Eq.(3.28.7)]{askey}
  (with $x$ is replaced by $\frac{x}{q}$ )
\begin{align}\label{4666vgf}
  D_{q^{-1}} h_n(x;q)= [n]_q h_{n-1}(\frac{x}{q};q),
\end{align} we get (${\ref{214dddfc}}$).
\end{proof}
\begin{example}~\par

\vskip .5 cm

\begin{itemize}
  \item If $m=n$ in  (${\ref{214dddfc}}$), then
\begin{align}\label{214587}
 & \int x^{n-2} (q^2x^2;q^2)_\infty h_n(x;q) d_qx=
\frac{ x^{n}(x^2;q^2)_\infty}{[n-1]_q} \left(\frac{h_n(\frac{x}{q};q)}{x}-q^{-1} h_{n-1}(\frac{x}{q};q)  \right), (n\neq1).
\end{align}
  \item It  we calculate the definite $q$-integral from $-1$ to $1$ in (${\ref{214587}}$), we obtain
\begin{align*}
 & \int_{-1}^{1} x^{n-2} (q^2x^2;q^2)_\infty h_n(x;q) d_qx= 0,
\end{align*}
which is expected since $( h_n(x;q))$ is orthogonal on $[-1,1]$.
\end{itemize}
\end{example}
\begin{proposition}
If $h_n(x;q)$ is  the  discrete $q$-Hermite I  polynomial of degree $n$, and $m$ $\in$ $\mathbb{N}_0$, then
  \begin{align*}
 & \left(q^{-n}[n]_q -q^{-m}[m]_q\right)   \int  (q^2x^2;q^2)_\infty h_n(x;q)h_m(x;q) d_qx \nonumber\\&=
 \frac{1}{q}(x^2;q^2)_\infty \left((1-q^m)h_{m-1}(\frac{x}{q};q)h_n(\frac{x}{q};q)-(1-q^n) h_{n-1}(\frac{x}{q};q) h_m(\frac{x}{q};q) \right).
  \end{align*}
\end{proposition}
\begin{proof}
  The proof follows by substituting with  $h(x)=h_m(x;q)$ and $y(x) = h_n(x;q) $ in (${\ref{mkl}}$).
\end{proof}
\begin{theorem}\label{59}
  If $\widetilde{h}_n(x;q)$ is the discrete $q$-Hermite II polynomial of degree $n$, then
  \begin{align}\label{2584}
  \int \frac{\widetilde{h}_n(x;q)}{(-x^2;q^2)_\infty} d_qx = \dfrac{-q^{1-n}(1-q)}{(-x^2;q^2)_\infty} \widetilde{h}_{n-1}(x;q),
\end{align}
and
\begin{align}\label{hbgddvsfa}
 & \int \frac{x^m}{(-x^2;q^2)_\infty} \left(\frac{[m]_q [m-1]_q}{ x^2}+\frac{q^{2m-1}[n-m]_q}{1-q}\right)\widetilde{h}_n(x;q) d_qx \nonumber\\&=\frac{x^m}{(-x^2;q^2)_\infty} \left(\frac{[m]_q}{ x}\widetilde{h}_n(x;q) -q^{m-n} [n]_q \widetilde{h}_{n-1}(x;q)\right).
\end{align}
\end{theorem}
\begin{proof}
The discrete $q$-Hermite II polynomials  is defined by $$\widetilde{h}_n(x;q)= x^n \,_2\phi_1\left( \begin{array}{cccc}
                     q^{-n},q^{-n+1} \\
                    0
                   \end{array}
\mid q^2;\frac{-q^2}{x^2}\right)$$ which satisfies the second-order $q$-difference equation, see \cite[Eq.(3.29.5)]{askey},
\begin{align}\label{kjdsmhg}
    \frac{1}{q} D_{q^{-1}} D_qy(x)- \dfrac{x}{1-q}D_q y(x) + \frac{ [n]_q}{1-q} y(x) =0.
\end{align}
  By comparing  Equation (${\ref{kjdsmhg}}$) with Equation (${\ref{mxlk}}$), we get
  \begin{align*}
    p(x)= - \dfrac{x}{1-q},\quad \quad r(x) = \frac{ [n]_q}{1-q} .
  \end{align*}
Then we get  $ f(x) =   \frac{1}{(-x^2;q^2)_\infty}$  is a solution of Equation (${\ref{frht5}}$).
  Substituting with  $h (x) =\frac{(1-q)}{[n]_q}$ into Equation (${\ref{mkrl}}$), and using \cite[Eq.(3.29.7)]{askey}
  (with $x$ is replaced by $\frac{x}{q}$ )
\begin{align}\label{4666vgfh}
  D_{q^{-1}}\widetilde{h}_n(x;q)=q^{1-n} [n]_q \widetilde{h}_{n-1}(x;q),
\end{align}
we get (${\ref{2584}}$).
Substituting with  $h (x) =x^m$ into Equation (${\ref{mkrl}}$), and using (${\ref{4666vgfh}}$) gives  (${\ref{hbgddvsfa}}$).
\end{proof}
\begin{remark}
It is worth noting that  (${\ref{2584}}$) is equivalent to
\begin{align*}
  D_q\left( w(x;q)\widetilde{h}_{n-1}(x;q) \right) = \frac{q^{n-1}}{1-q} w(x;q)\widetilde{h}_{n}(x;q),
\end{align*}
where $ w(x;q)= \frac{1}{(-x^2;q^2)}$, see \cite[Eq.(3.29.9)]{askey}.\\
Moreover,\begin{equation*}
\int_{-c}^{c} \frac{\widetilde{h}_n(x;q)}{(-x^2;q^2)_\infty} d_qx =\left\{\begin{array}{lc}
  0, & \text{ if $n$ is odd;} \\\\
  \frac{-2q^{1-n}(1-q)}{(-c^2;q^2)_\infty} \widetilde{h}_{n-1}(c;q), & \text{ if $n$ is even.}
\end{array}\right.
\end{equation*}
  \end{remark}
\begin{proposition}
If $\widetilde{h}_n(x;q)$ is the discrete $q$-Hermite II polynomial of degree $n$, and $m$ $\in$ $\mathbb{N}_0$, then
  \begin{align*}
  &\left([n]_q-[m]_q\right)  \int \frac{1}{(-x^2;q^2)_\infty} \widetilde{h}_n(x;q)\widetilde{h}_m(x;q)d_qx \\&= \dfrac{q}{(-x^2;q^2)_\infty}\left(q^{-m}(1-q^m)\widetilde{h}_{m-1}(x;q)\widetilde{h}_n(x;q)-q^{-n}(1-q^n)  \widetilde{h}_{n-1}(x;q)\widetilde{h}_m(x;q) \right) .
  \end{align*}
\end{proposition}
\begin{proof}
  The proof follows by substituting with  $\widetilde{h}_m(x;q)$ and $y(x) = \widetilde{h}_n(x;q) $ in (${\ref{mkrl}}$).
  \end{proof}
\begin{theorem}
If  $Ai_{\sqrt{q}}(x)$ is the  $\sqrt{q}$-Airy function, then
\begin{align}\label{25844}
\int  \frac{Ai_{\sqrt{q}}(x)}{(q^{\frac{-1}{2}}x^2;q^2)_\infty} d_qx = \frac{\sqrt{q}(1-q)^2}{(q^{\frac{-1}{2}}x^2;q^2)_\infty}D_{q^{-1}} Ai_{\sqrt{q}}(x),
\end{align}
and
\begin{align}\label{kmjnhbgvfd}
  & \int  \frac{x^{m} }{(q^{\frac{-1}{2}}x^2;q^2)_\infty} \left(\frac{[m]_q [m-1]_q}{x^2}-\frac{q^{2m-\frac{3}{2}}}{(1-q)^2}\right)Ai_{\sqrt{q}}(x)d_qx =\nonumber\\& \frac{x^{m
   }}{(q^{\frac{-1}{2}}x^2;q^2)_\infty}\left(\frac{[m]_q Ai_{\sqrt{q}}(x)}{x}-q^{m-1}D_{q^{-1}} Ai_{\sqrt{q}}(x)\right).
 \end{align}
\end{theorem}
\begin{proof}
The  $q$-Airy function, $Ai_q(x)= \sum_{n=0}^{\infty} \dfrac{q^{\frac{n(n-1)}{2}}}{(q^2;q^2)_n} x^n$, satisfies the second-order $q$-difference Equation, see \cite{mou},
\begin{align}\label{kjdsmhrfgfg}
    \frac{1}{q^2} D_{q^{-2}} D_{q^2}y(x)+ \dfrac{x}{q(1-q^2)}D_{q^2} y(x) -\frac{ 1}{q(1-q^2)^2} y(x) =0.
\end{align}
 By comparing  Equation (${\ref{kjdsmhrfgfg}}$) with Equation (${\ref{mkrl}}$), we obtain
  \begin{align*}
    p(x)= \dfrac{x}{\sqrt{q}(1-q)}, \quad \quad r(x) =  -\frac{ 1}{\sqrt{q}(1-q)^2} .
  \end{align*}
Hence $f(x) =   \frac{1}{(q^{\frac{-1}{2}}x^2;q^2)_\infty}$
is a solution of Equation (${\ref{frht5}}$).
Substituting with $ h(x)=\sqrt{q}(1-q)^2$ and $ x^m$
into Equation ({\ref{mkrl}}) gives the $q$-integrals  (${\ref{25844}}$) and (${\ref{kmjnhbgvfd}}$),  respectively.
\end{proof}
\section{Extensions of Lagrangian Method for NHSOqDE}
\subsection{Indefinite $q$-integrals from an
inhomogeneous $q$-difference equation}
In this section we extend the Lagrangian method to include the second-order inhomogeneous $q$-difference equations.
\begin{theorem}\label{jnhsvvffss}
Let $p(x)$, $r(x)$ and $g(x)$ be continuous functions at zero. Let  $y(x)$ be any solution  of the second-order  $q$-difference equation
\begin{equation}\label{mnm}
  \frac{1}{q}D_{q^{-1}} D_q y(x) +  p(x)D_{q^{-1}}y(x) + r(x) y(x) =g(x).
\end{equation}
\text{ Then}
\begin{align}\label{mkvvl}
  &\int f(x) \left( \frac{1}{q} D_{q^{-1}} D_qh(x)+   p(x) D_{q^{-1}} h(x) +r(x) h(x)\right) y(x) d_qx = \nonumber\\&
   f(x/q)\left( y(x/q) D_{q^{-1}}h(x) - h(x/q) D_{q^{-1}}y(x) \right)+\int f(x) h(x) g(x) d_qx,
\end{align}
where $ h(x) $ is an arbitrary function  and $ f(x)$  is a solution of the first order  $q$-difference Equation $({\ref{fht}})$.
\end{theorem}
\begin{proof}
By applying the  $q$-integration by part rule (${\ref{555}}$)
  \begin{align*}
&\frac{1}{q}\int f(x)y(x)  D_{q^{-1}} D_qh(x) d_qx  =   (D_qh)(x/q)f(x/q) y(x/q)- \frac{1}{q} \int (D_qh)(x/q)D_{q^{-1}}(f(x) y(x)) d_qx.
\end{align*}
Since by The $q$-product rule, we get
\begin{align*}
  \frac{1}{q}\int f(x)y(x)  D_{q^{-1}} D_qh(x) d_qx &= D_{q^{-1}}h(x) f(x/q) y(x/q)- \frac{1}{q} \int (D_qh)(x/q) f(x/q)D_{q^{-1}} y(x)d_qx  \\&-\frac{1}{q} \int (D_qh)(x/q)D_{q^{-1}}(f(x))y(x)d_qx.
\end{align*}
Applying the  $q$-integration by part rule (${\ref{555}}$), we get
\begin{align*}
  \frac{-1}{q} \int (D_qh)(x/q) f(x/q)D_{q^{-1}} y(x)d_qx &= - h(x/q) f(x/q) D_{q^{-1}} y(x)+\int  h(x) f(x) D_q D_{q^{-1}}y(x) d_qx \\&+ \int h(x) D_{q^{-1}} y(x) D_q f(x/q)d_qx.
\end{align*}
Using Equation (${\ref{mxlk}}$) and (${\ref{fht}}$), we obtain
\begin{align*}
 &\frac{1}{q}\int f(x)y(x)  D_{q^{-1}} D_qh(x) d_qx = D_{q^{-1}}h(x) f(x/q) y(x/q) -h(x/q)f(x/q)D_{q^{-1}}y(x)   \\& - \int h(x) f(x) r(x) y(x) d_q(x)  -\int D_{q^{-1}}h(x)f(x) p(x) y(x)d_qx+\int f(x) h(x) g(x) d_qx.
\end{align*}
Then, we get the desired result.
\end{proof}
\begin{theorem}
Let $p(x)$, $r(x)$ and $g(x)$ be continuous functions at zero. Let  $y(x)$ be any solution  of the second-order  $q$-difference equation
\begin{equation}\label{mnkm}
  \frac{1}{q}D_{q^{-1}} D_q y(x) +  p(x)D_q y(x) + r(x) y(x) =g(x).
\end{equation}
\text{ Then}
\begin{align}\label{mkvvbgl}
  &\int f(x) \left( \frac{1}{q} D_{q^{-1}} D_qh(x)+   p(x) D_q h(x) +r(x) h(x)\right) y(x) d_qx = \nonumber\\&
   f(x)\left( y(x) D_{q^{-1}}h(x) - h(x) D_{q^{-1}}y(x) \right)+\int f(x) h(x) g(x) d_qx,
\end{align}
where $ h(x) $ is an arbitrary function  and $ f(x)$  is a solution of the first order  $q$-difference Equation $({\ref{frht5}})$.
\end{theorem}
\begin{proof}
The proof follows similarity as the proof of Theorem {$\ref{jnhsvvffss}$}.
\end{proof}
\begin{theorem}\label{theorem4}
  For $ \Re (\nu) > \frac{-1}{2}$,  $ \lambda \in \mathbb{C}$, and $\Re (\mu+\nu) > 0$,  we get the $q$-integrals
  \begin{align}\label{222}
 & \int x^{\mu+1} \left(A_{\mu,\nu} x^{-2}+q^{-\nu-1}\lambda^2\right)  H_\nu^{(3)}(\lambda x;q^2) d_qx = \dfrac{\lambda^{1+\nu}C_\nu(q)x^{\nu+\mu+1}}{q^{\nu+1}[\nu+\mu+1]_q } \nonumber\\& + x^{\mu+1} \left(\frac{q^{-\mu}[\mu]_q  }{x}H_\nu^{(3)}(\frac{\lambda x}{q};q^2)-q^{-\mu-1}D_{q^{-1}}  H_\nu^{(3)}(\lambda x;q^2) \right),
\end{align}
or equivalently,
\begin{align}\label{bhnnndujdj3}
 & \int x^{\mu+1} \left(q^{\mu+1} A_{\mu,\nu} x^{-2}+q^{-\nu+\mu}\lambda^2\right)  H_\nu^{(3)}(\lambda x;q^2) d_qx \nonumber\\&= q x^{\mu+1} \left(\frac{[\mu]_q -[\nu]_q }{x}H_\nu^{(3)}(\frac{\lambda x}{q};q^2)+\frac{\lambda}{q} H_{\nu+1}^{(3)}(\lambda x;q^2) \right)\nonumber\\&+\dfrac{\lambda^{1+\nu}C_\nu(q) [\nu-\mu]_q}{q^{\nu-\mu}[\nu+\mu+1]_q[2\nu+1]_q}x^{\nu+\mu+1} .
\end{align}
In particular,
\begin{align}\label{2145gg}
& \int x^{\nu+1}  H_\nu^{(3)}(\lambda x;q^2) d_qx =  \frac{x^{\nu+1}}{\lambda}  H_{\nu+1}^{(3)}(\lambda x;q^2) ,
\end{align}
and
\begin{align}\label{32n5}
& \int x^{1-\nu}  H_\nu^{(3)}(\lambda x;q^2) d_qx =  \frac{-q^\nu}{\lambda} x^{1-\nu} H_{\nu-1}^{(3)}(\frac{\lambda x}{q};q^2) +  \lambda^{\nu-1}C_\nu(q) x,
\end{align}
where $A_{\mu,\nu}$ is defined as in Theorem {$\ref{thmnbs}$} and $C_\nu(q) := \dfrac{(1+q)^{-\nu+1}}{\Gamma_{q^2}(\frac{1}{2})\Gamma_{q^2}(\nu+\frac{1}{2}) }$.
\end{theorem}
\begin{proof}
The  $q$-Struve function $H_\nu^{(3)}(\lambda x;q^2)$ associated with the  third Jackson $q$-Bessel Equation
 satisfies the  $q$-difference  equation, see \cite[Eq.(21)]{struve}
  \begin{align}\label{mnn}
&\frac{1}{q} D_{q^{-1}}D_q y(x) + \frac{1}{qx} D_{q^{-1}} y(x) +\Big(q^{-\nu-1} \lambda ^{2}  - q^{-\nu} [\nu]^2 _q x^{-2} \Big) y(x)=\dfrac{\lambda^{\nu+1} x^{\nu-1}(1+q)^{-\nu+1}}{q^{\nu+1}\Gamma_{q^2}(\frac{1}{2})\Gamma_{q^2}(\nu+\frac{1}{2}) }.
\end{align}
  Thus,
\begin{align*}
  g(x)=\dfrac{\lambda^{\nu+1} x^{\nu-1}(1+q)^{-\nu+1}}{q^{\nu+1}\Gamma_{q^2}(\frac{1}{2})\Gamma_{q^2}(\nu+\frac{1}{2}) } \quad \quad and \quad  f(x)=x.
\end{align*}
Therefore,
\begin{align}\label{sred}
 \int f(x) h(x) g(x) d_qx = \left(\dfrac{\lambda}{q}\right)^{\nu+1} C_\nu(q) \int x^{\nu} h(x) d_qx .
\end{align}
For $h(x) = x^{\mu}$, Equation (${\ref{sred}}$) will be
\begin{align}\label{jnuuqqww}
 \int f(x) h(x) g(x) d_qx = \dfrac{\lambda^{\nu+1}C_\nu(q)}{q^{\nu+1} [\nu+\mu+1]_q}x^{\nu+\mu+1} .
\end{align}
Substituting with the $q$-integral in (${\ref{jnuuqqww}}$) into Equation (${\ref{mkvvl}}$) gives the   $q$-integral (${\ref{222}}$). The proof of (${\ref{bhnnndujdj3}}$) follows
by using \cite[Eq.(18)]{struve} (with $x$ is replaced by $\frac{x}{q}$ ) to obtain
\begin{align*}
  &x D_{q^{-1}} H_\nu^{(3)}(\lambda x;q^2)-q[\nu]_q H_\nu^{(3)}(\frac{\lambda x}{q};q^2)=C_\nu(q)(\lambda x)^{\nu+1}-\lambda x  H_{\nu+1}^{(3)}(\lambda x;q^2),\end{align*}
and  substituting in Equation (${\ref{222}}$). Substituting with $\nu=\mu$ yields (${\ref{2145gg}}$). Substituting with $\mu=-\nu$, in Equation (${\ref{bhnnndujdj3}}$) and using
\cite[Eq.(15)]{struve}(with $x$ is replaced by $\frac{x}{q}$ and multiplying with  $x^{1-\nu}$ ) to obtain
\begin{align*}
& x^{1-\nu} H_{\nu-1}^{(3)}(\frac{\lambda x}{q};q^2)+q^{\nu} x^{1-\nu} H_{\nu+1}^{(3)}(\lambda x;q^2)= \frac{q[2\nu]_qx^{-\nu}}{\lambda} H_\nu^{(3)}(\frac{\lambda x}{q};q^2)+\frac{(\lambda q)^{\nu} C_\nu(q)}{ [2\nu+1]_q} x ,
\end{align*}
we get (${\ref{32n5}}$), and completes the proof.
\end{proof}
\begin{theorem}
 For $ \Re  (\nu) > \frac{-1}{2} $. Let
 \begin{align*}
 & \tilde{A}_{n,\nu}:= q^{\frac{1}{2}(3n+\nu+3)}A_{n,\nu}, \quad \quad \hat{A}_{n,\nu} = q^{\frac{1}{2}(3n+\nu)}A_{n,\nu},\\&
  O_{n,\nu}(x) := \frac{q^{\frac{1}{2}(n+\nu+3)}[n]}{x} \sin(q^{\frac{-1}{2}(\nu+n+1)}x;q)+  \cos( q^{\frac{-1}{2}(\nu+n+2)} x;q),\\&
\tilde{O}_{n,\nu}(x) := \frac{q^{\frac{1}{2}(n+\nu)}[n]}{x} \cos(q^{\frac{-1}{2}(\nu+n+1)}x;q)-  \sin( q^{\frac{-1}{2}(\nu+n+2)} x;q).
 \end{align*}
 Then
\begin{align*}
&\int  \left(  [2n+1]_q  x^n \cos(q^{\frac{-1}{2}(n+\nu+2)}x;q)+x^{n-1} \tilde{A}_{n,\nu}  \sin(q^{\frac{-1}{2}(n+\nu+1)}x;q)  \right)  H_\nu^{(3)}( x;q^2) d_qx \nonumber \\&= x^{n+1} O_{n,\nu}(x) H_\nu^{(3)}( \frac{x}{q};q^2)-q^{\frac{1}{2}(n+\nu+1)}x^{n}\sin(q^{\frac{-1}{2}(n+\nu+3)}x;q) D_{q^{-1}}H_\nu^{(3)}( x;q^2)\\&+ \frac{q^{n-\nu} C_\nu(q) x^{n+\nu+2}}{[n+\nu+2]_q} \,_2\phi_2(0,q^{\nu+n+2};q^{3},q^{\nu+n+4};q^2,q^{\nu+m+1}(1-q)^2 x^2),
\end{align*}
and
\begin{align*}
&\int  \left(  - [2n+1]_q  x^n \sin(q^{\frac{-1}{2}(n+\nu+2)}x;q)+x^{n-1} \hat{A}_{n,\nu}  \cos(q^{\frac{-1}{2}(n+\nu+1)}x;q)  \right)  H_\nu^{(3)}( x;q^2) d_qx \nonumber \\&= x^{n+1}  \tilde{O}_{n,\nu}(x) H_\nu^{(3)}( \frac{x}{q};q^2)-q^{n+\frac{\nu}{2}-1}x^{n}\cos(q^{\frac{-1}{2}(n+\nu+3)}x;q) D_{q^{-1}}H_\nu^{(3)}( x;q^2)\\&+ \frac{q^{\frac{-1}{2}(\nu-3n+2)}C_\nu(q) x^{\nu+n+1}}{[\nu+n+1]_q} \,_2\phi_2(0,q^{\nu+n+1};q^{3},q^{\nu+n+3};q^2,q^{-n-\nu}(1-q)^2 x^2),
\end{align*}
where $A_{n,\nu}$  is defined as in Theorem {$\ref{thmnbs}$} and $C_\nu(q)$ is defined as  in Theorem {$\ref{theorem4}$}.
\end{theorem}
\begin{proof}
Substituting with $$h(x)= x^n \sin(q^{\frac{-1}{2}(n+\nu+1)}x;q)$$, and $h(x)=x^n\cos(q^{\frac{-1}{2}(n+\nu+1)}x;q)$ , respectively, with $\lambda=1$  into Equation ({\ref{mkvvl}}), we get the desired result.
\end{proof}
\begin{theorem}\label{thm7}
 For $ \Re \, (\nu) > \frac{-1}{2}$,  $ \lambda \in \mathbb{C} $, $x \in \mathbb{C} $ and $\Re \, (\mu+\nu) > 0$,  we get the $q$-integrals
\begin{align}\label{333}
 &\int \dfrac{x^{\mu+1}}{(- x^2  \lambda^2(1-q)^2;q^2)_\infty} \left( A_{\mu,\nu} x^{-2}+q^{\mu}\lambda^2\right) H_\nu^{(2)}(\lambda x|q^2) d_qx \nonumber\\&= \frac{x^{\mu+1}}{(- x^2  \lambda^2(1-q)^2;q^2)_\infty} \left(\frac{q^{-\mu}[\mu]_q}{x}H_\nu^{(2)}(\lambda x|q^2)-\frac{1}{q}  D_{q^{-1}}H_\nu^{(2)}(\lambda x|q^2) \right)\nonumber\\&+\left(\dfrac{\lambda}{q}\right)^{\nu+1} C_\nu(q) \int \dfrac{x^{\nu+\mu}}{(- x^2  \lambda^2(1-q)^2;q^2)_\infty}  d_qx.
\end{align}
In particular,
\begin{align}\label{214bgd}
  &\int \dfrac{x^{\nu+1}}{(- x^2  \lambda^2(1-q)^2;q^2)_\infty} H_\nu^{(2)}(\lambda x|q^2) d_qx =\frac{ q x^{\nu+1}}{\lambda(- x^2  \lambda^2(1-q)^2;q^2)_\infty}H_{\nu+1}^{(2)}(\lambda x|q^2)
\nonumber\\& - \dfrac{\lambda^{\nu-1}}{q^{2\nu+1} } C_\nu(q) \left(\frac{ x^{2\nu+1}}{[2\nu+1]_q(- x^2  \lambda^2(1-q)^2;q^2)_\infty}-\int \dfrac{x^{2\nu}}{(- x^2  \lambda^2(1-q)^2;q^2)_\infty}  d_qx \right),\end{align}
and
\begin{align}\label{21lmhf}
 &\int \dfrac{x^{1-\nu}}{(- x^2 \lambda^2(1-q)^2;q^2)_\infty}  H_\nu^{(2)}(\lambda x|q^2) d_qx  = \frac{x^{1-\nu}}{q\lambda(- x^2  \lambda^2(1-q)^2;q^2)_\infty} H_{\nu-1}^{(2)}(\lambda x|q^2) \nonumber \\& +\dfrac{1}{q} \lambda^{\nu-1} C_\nu(q) \int \dfrac{1}{(- x^2  \lambda^2(1-q)^2;q^2)_\infty}
   d_qx,
\end{align}
where $A_{\mu,\nu}$ and $C_\nu(q)$  are defined as in Theorem {$\ref{thmnbs}$} and   Theorem {$\ref{theorem4}$}, respectively.
\end{theorem}
\begin{proof}
The  $q$-Struve function $H_\nu^{(2)}(\lambda x|q^2)$ associated with the second Jackson $q$-Bessel Equation(${\ref{mnbvxx}}$) satisfies the  $q$-difference  equation  \cite[Eq.(40)]{struve}
  \begin{align}\label{mnnv}
&\frac{1}{q} D_{q^{-1}}D_q y(x) + \left(\dfrac{1}{x}-q \lambda^2 x (1-q)\right) D_{q} y(x) +
\left(q  \lambda ^{2}   -\frac{ q^{1-\nu} [\nu]^2_q }{x^2}\right) y(x)\nonumber\\&=\dfrac{\lambda^{\nu+1} C_\nu(q)  }{q^\nu }x^{\nu-1}.
\end{align}
 Thus,
\begin{align*}
  g(x)=\dfrac{\lambda^{\nu+1} C_\nu(q) }{q^\nu  }x^{\nu-1} \quad \quad and  \quad f(x)=\dfrac{x}{(- x^2  \lambda^2(1-q)^2;q^2)_\infty},
\end{align*}
with $h(x) = x^{\mu}$, we get
\begin{align*}
 \int f(x) h(x) g(x) d_qx = \dfrac{\lambda^{\nu+1} C_\nu(q)}{q^\nu }\int \dfrac{x^{\nu+\mu}}{(- x^2  \lambda^2(1-q)^2;q^2)_\infty} d_qx .
\end{align*}
Substituting  into Equation (${\ref{mkvvbgl}}$)  gives the  $q$-integral (${\ref{333}}$).
Using \cite[Eq.(39)]{struve} (with $x$ is replaced by $\frac{x}{q}$ ) to obtain
\begin{align}\label{impot}
  & D_{q^{-1}}H_\nu^{(2)}(\lambda x|q^2)-\frac{q^{1-\nu} [\nu]_q}{x}H_\nu^{(2)}(\lambda x|q^2)=\frac{\lambda^{\nu+1}(1+q)^{-\nu}x^\nu}{q^\nu \Gamma_{q^2}(\frac{1}{2})\Gamma_{q^2}(\nu+\frac{3}{2})}-\lambda q^{\nu+2}H_{\nu+1}^{(2)}(\lambda x|q^2),
\end{align}
Equation (${\ref{333}}$) can be represented as
\begin{align}\label{jnnndhnndnndn}
 &\int \dfrac{x^{\mu+1}}{(- x^2  \lambda^2(1-q)^2;q^2)_\infty} \left( A_{\mu,\nu} x^{-2}+q^{\mu}\lambda^2\right) H_\nu^{(2)}(\lambda x|q^2) d_qx \nonumber\\&= \frac{x^{\mu+1}}{(- x^2  \lambda^2(1-q)^2;q^2)_\infty} \left(\frac{q^{-\mu}[\mu]_q-q^{1-\nu}[\nu]}{x}H_\nu^{(2)}(\lambda x|q^2)+ q^{\nu+1} \lambda H_{\nu+1}^{(2)}(\lambda x|q^2) \right)\nonumber\\&-\left(\dfrac{\lambda}{q}\right)^{\nu+1} C_\nu(q) \left( \frac{x^{\nu+\mu+1}}{[2\nu+1]_q (- x^2  \lambda^2(1-q)^2;q^2)_\infty}-\int \dfrac{x^{\nu+\mu}}{(- x^2  \lambda^2(1-q)^2;q^2)_\infty}  d_qx \right).
\end{align}
The proof of $({\ref{214bgd}})$ follows by
 substituting  with $\nu=\mu$ in $({\ref{jnnndhnndnndn}})$. Substituting with $\mu=-\nu$, in Equation (${\ref{jnnndhnndnndn}}$), using (${\ref{impot}}$) and
\cite{struve}(with $x$ is replaced by $\frac{x}{q}$ )
 \begin{align*}
  \frac{[\nu]_q}{x} H_\nu^{(2)}(\lambda x|q^2)+\frac{1}{q}D_{q^{-1}}H_\nu^{(2)}(\lambda x|q^2)=\lambda q^{-(\nu+1)}H_{\nu-1}^{(2)}(\lambda x|q^2),
 \end{align*}
 to obtain
 \begin{align*}
 & \lambda q^{\nu+1}x^{1-\nu} H_{\nu+1}^{(2)}(\lambda x|q^2)-\frac{q^{-\nu}[2\nu]}{x^\nu} H_\nu^{(2)}(\lambda x|q^2)=\\&\frac{\lambda^{\nu+1} C_\nu(q) x}{q^{\nu+1}[2\nu+1]_q  }+\lambda q^{-\nu-1} x^{1-\nu} H_{\nu-1}^{(2)}(\lambda x|q^2),
 \end{align*}
we get (${\ref{21lmhf}}$), and completes the proof.
\end{proof}
\begin{remark}
   For $ \Re  (\nu) > -\frac{1}{2}$,  $ \lambda \in \mathbb{C} $, $\mid \lambda x (1-q)\mid < 1 $ and $\Re  (\mu+\nu) > 0$, Equation (${\ref{333}}$) can be represented as
\begin{align*}
&\int \dfrac{x^{\mu+1}}{(- x^2  \lambda^2(1-q)^2;q^2)_\infty} \left( A_{\mu,\nu} x^{-2}+q^{\mu}\lambda^2\right) H_\nu^{(2)}(\lambda x|q^2) d_qx \nonumber\\&= \frac{x^{\mu+1}}{(- x^2  \lambda^2(1-q)^2;q^2)_\infty} \left(\frac{q^{-\mu}[\mu]_q}{x}H_\nu^{(2)}(\lambda x|q^2)-\frac{1}{q}  D_{q^{-1}}H_\nu^{(2)}(\lambda x|q^2) \right)\nonumber\\&+\dfrac{\lambda^{\nu+1} C_\nu(q) x^{\nu+\mu+1}}{q^{\nu+1}  [\nu+\mu+1]_q}\,_2\phi_1(0,q^{\nu+\mu+1};q^{\nu+\mu+3};q^2,-x^2  \lambda^2(1-q)^2).
\end{align*}
\end{remark}
\begin{theorem}
 For $ \Re  (\nu) > -\frac{1}{2} $,  let
 \begin{align*}
   &B_{n,\nu}(x) := \frac{q^{-n}[n]_q}{x}\S_q (q^{n+\frac{1}{2}}x) +q^{\frac{3}{2}}\c_q(q^{n+\frac{1}{2}}x),\\&
 \tilde{B}_{n,\nu}(x) := \frac{q^{-n}[n]_q}{x}\c_q (q^{n+\frac{1}{2}}x) -q^{\frac{3}{2}}\S_q(q^{n+\frac{1}{2}}x).
 \end{align*}
 Then
 \begin{align*}
   &\int \Big(x^{n-1} A_{n,\nu}    \S_q(q^{n+\frac{3}{2}}x)  +x^n C_{n,\nu}(x)  \c_q(q^{n+\frac{3}{2}}x) \Big)  \frac{H_\nu^{(2)}( x |q^2)}{(- x^2  (1-q)^2;q^2)_\infty}  d_qx \nonumber\\&=\frac{x^{n+1}}{(- x^2  (1-q)^2;q^2)_\infty}\Big( B_{n,\nu}(x) H_\nu^{(2)}(  x |q^2)-q  \S_q (q^{n+\frac{1}{2}}x) D_{q^{-1}} H_\nu^{(2)}( x|q^2)\Big)
   \\&+q^{1-\nu} C_\nu(q) \int \dfrac{x^{\nu+n}\S_q(q^{n+\frac{1}{2}}x)}{(- x^2  (1-q)^2;q^2)_\infty}  d_qx,
  \end{align*}
  and
  \begin{align*}
   &\int \Big(x^{n-1} A_{n,\nu}    \c_q(q^{n+\frac{3}{2}}x)  -x^n C_{n,\nu}(x)  \S_q(q^{n+\frac{3}{2}}x) \Big)  \frac{H_\nu^{(2)}( x |q^2)}{(- x^2  (1-q)^2;q^2)_\infty}  d_qx \nonumber\\&=\frac{x^{n+1}}{(- x^2  (1-q)^2;q^2)_\infty}\Big(\tilde{B}_{n,\nu}(x) H_\nu^{(2)}(  x |q^2)-q  \c_q (q^{n+\frac{1}{2}}x) D_{q^{-1}} H_\nu^{(2)}( x|q^2)\Big)
   \\&+q^{1-\nu} C_\nu(q)\int \dfrac{x^{\nu+n}\c_q(q^{n+\frac{1}{2}}x)}{(- x^2  (1-q)^2;q^2)_\infty}   d_qx,
  \end{align*}
  where $A_{n,\nu}$, $C_{n,\nu}(x)$ and $C_\nu(q)$ are defined as in Theorem {$\ref{thmnbs}$},  Theorem {$\ref{2514b4}$} and Theorem {$\ref{theorem4}$}, respectively.
  \end{theorem}
  \begin{proof}
Substituting   with $h(x)= x^n \S_q(q^{n+\frac{1}{2}}x)$, and $h(x)=  x^n \c_q(q^{n+\frac{1}{2}}x)$, respectively, into Equation ({\ref{mkvvbgl}}) with $\lambda=1$,   we get the   desired result.
  \end{proof}
\begin{theorem}\label{thm54}
  For $ \Re  {\nu} > \frac{-1}{2}$,  $ \lambda \in \mathbb{C}$, and $\Re  (\mu+\nu) > 0$,
  \begin{align*}
&\int x^{\mu+1}(- x^2  \lambda^2(1-q)^2;q^2)_\infty \left(A_{\mu,\nu} x^{-2}+q^{-(\mu+2)}\lambda^2\right) H_\nu^{(1)}(\lambda x|q^2) d_qx \nonumber \\& =  x^{\mu+1}(- \frac{\lambda^2}{q^2}(1-q)^2 x^2;q^2)_\infty \left(\frac{q^{-\mu}[\mu]_q}{x}H_\nu^{(1)}(\frac{\lambda x}{q}|q^2)-q^{-\mu-1} D_{q^{-1}}H_\nu^{(1)}(\lambda x|q^2) \right)\nonumber \\&+\dfrac{\lambda^{\nu+1}  C_\nu(q) x^{\nu+\mu+1}}{q^{\nu+1} [\nu+\mu+1]_q} \,_1\phi_1(q^{\nu+\mu+1};q^{\nu+\mu+3};q^2,-x^2  \lambda^2(1-q)^2).
\end{align*}
In particular,
\begin{align*}
&\int x^{1+\nu}(- x^2  \lambda^2(1-q)^2;q^2)_\infty  H_\nu^{(1)}(\lambda x|q^2) d_qx =\frac{ {(qx)}^{1+\nu}}{\lambda}(- \frac{\lambda^2}{q^2}(1-q)^2 x^2;q^2)_\infty H_{\nu+1}^{(1)}(\frac{\lambda x}{q}|q^2)\nonumber\\&+ \dfrac{q \lambda^{\nu-1} C_\nu(q)x^{2\nu+1}}{[2\nu+1]_q } \Big(\,_1\phi_1(q^{2\nu+1};q^{2\nu+3};q^2,-x^2  \lambda^2(1-q)^2) \nonumber- (-  \frac{\lambda^2}{q^2}(1-q)^2 x^2;q^2)_\infty \Big),
\end{align*}
and
\begin{align*}
&\int x^{1-\nu}(- x^2  \lambda^2(1-q)^2;q^2)_\infty H_\nu^{(1)}(\lambda x|q^2) d_qx =\frac{ - {(qx)}^{1-\nu}}{\lambda}(- \frac{\lambda^2}{q^2}(1-q)^2 x^2;q^2)_\infty  H_{\nu-1}^{(1)}(\frac{\lambda x}{q}|q^2)\nonumber\\&+q^{-2\nu+1}\lambda^{\nu-1} C_\nu(q) x \,_1\phi_1(q;q^{3};q^2,-x^2  \lambda^2(1-q)^2),
\end{align*}
where $A_{\mu,\nu}$ and $C_\nu(q)$ are defined as in Theorems {$\ref{thmnbs}$} and  {$\ref{theorem4}$}, respectively.
\end{theorem}
\begin{theorem}
For $ \Re  (\nu) > \frac{-1}{2}$,  let
\begin{align*}
  &E_{n,\nu}(x):= \frac{q^{-n}[n]_q}{x} \sin_q(q^{-(n+\frac{3}{2})}x) +q^{-(n+\frac{3}{2})}\cos_q(q^{-(n+\frac{3}{2})}x),\\&
   \tilde{E}_{n,\nu}(x):= \frac{q^{-n}[n]_q}{x} \cos_q(q^{-(n+\frac{3}{2})}x) -q^{-(n+\frac{3}{2})}\sin_q(q^{-(n+\frac{3}{2})}x).
\end{align*}
Then
\begin{align*}
   &\int (- x^2  (1-q)^2;q^2)_\infty  \Big( A_{n,\nu}  x^{n-1} \sin_q(q^{-(n+\frac{3}{2})}x)  + K_{n,\nu} x^{n} \cos_q (q^{-(n+\frac{3}{2})}x) \Big)  H_\nu^{(1)}(   x |q^2) d_qx \nonumber \\&=x^{n+1}(- \frac{x^2 }{q^2}(1-q)^2;q^2)_\infty  \Big( E_{n,\nu}(x) H_\nu^{(1)}(\frac{ x}{q}|q^2)-q^{-(n+1)} \sin_q (q^{-(n+\frac{3}{2})}x) D_{q^{-1}}H_\nu^{(1)}( x|q^2)\Big)\\&+ q^{-\nu-1} C_\nu(q) \int x^{\nu+n}(- x^2  (1-q)^2;q^2)_\infty \sin_q (q^{-n-\frac{1}{2}}x)  d_qx,
  \end{align*}
and
\begin{align*}
   &\int (- x^2  (1-q)^2;q^2)_\infty  \Big( A_{n,\nu}  x^{n-1} \cos_q(q^{-(n+\frac{3}{2})}x)  - K_{n,\nu} x^{n} \sin_q (q^{-(n+\frac{3}{2})}x) \Big)  H_\nu^{(1)}(   x |q^2) d_qx \nonumber \\&=x^{n+1}(- \frac{x^2 }{q^2}(1-q)^2;q^2)_\infty  \Big( \tilde{E}_{n,\nu}(x) H_\nu^{(1)}(\frac{ x}{q}|q^2)-q^{-(n+1)} \cos_q (q^{-(n+\frac{3}{2})}x) D_{q^{-1}}H_\nu^{(1)}( x|q^2)\Big)\\&+q^{-\nu-1} C_\nu(q) \int x^{\nu+n}(- x^2  (1-q)^2;q^2)_\infty \cos_q (q^{-n-\frac{1}{2}}x)  d_qx,
  \end{align*}
  where $A_{n,\nu}$, $K_{n,\nu}$  and $C_\nu(q)$ are defined as in Theorems {$\ref{thmnbs}$}, {$\ref{nbhdsnjjnsb}$} and {$\ref{theorem4}$}, respectively.
\end{theorem}
\subsection{$q$-Integrals of the function $y (x)$ itself}
\begin{theorem}\label{x}
  Let $p(x)$ and $r(x)$ be continuous functions at zero. Let  $y(x)$ be any solution  of $({\ref{mxlk}})$. If $h(x)$ is a solution of the inhomogeneous equation
\begin{align}\label{xyy}
 \frac{1}{q} D_{q^{-1}} D_qh(x)+   p(x) D_{q^{-1}} h(x) +r(x) h(x) =\frac{1}{f(x)},
\end{align}
then
\begin{align*}
  \int   y(x) d_qx \nonumber &=
   f(x/q)\Big (y(x/q) D_{q^{-1}}h(x) - h( x/q) D_{q^{-1}}y(x) \Big)\\&= f(x/q)\Big (y(x) D_{q^{-1}}h(x) - h(x) D_{q^{-1}}y(x) \Big),
\end{align*}
where  $ f(x)$  is a solution of $({\ref{fht}})$.
\end{theorem}
\begin{proof}
Let $h(x)$ be a solution of the inhomogeneous Equation (${\ref{xyy}}$).
 By substituting with (${\ref{xyy}}$) in the  $q$-integral (${\ref{mkl}}$), we obtain the required result.
\end{proof}
\begin{theorem}\label{xq}
  Let $p(x)$ and $r(x)$ be continuous functions at zero. Let  $y(x)$ be any solution  of $({\ref{mx}})$. If $h(x)$ is a solution of the inhomogeneous equation
\begin{align}\label{xyly}
 \frac{1}{q} D_{q^{-1}} D_qh(x)+   p(x) D_q h(x) +r(x) h(x) =\frac{1}{f(x)},
\end{align}
then
\begin{align*}
 & \int   y(x) d_qx \nonumber =
   f(x)\Big (y(x) D_{q^{-1}}h(x) - h(x) D_{q^{-1}}y(x) \Big),
\end{align*}
where  $ f(x)$  is a solution of $({\ref{frht5}})$.
\end{theorem}
\begin{proof}
Let $h(x)$ be a solution of the inhomogeneous Equation (${\ref{xyly}}$).
 By substituting  with (${\ref{xyly}}$) in the  $q$-integral (${\ref{mkrl}}$), we obtain the desired result.
\end{proof}
\begin{theorem}\label{thmnb1}
  For $n \in \mathbb{N}$, we have
\begin{align*}
  \int p_n(x|q) d_qx = \frac{x }{q}(1-x) \,_2\phi_1(q^{-n+1} , q^{n+2};q^2;q,x).
\end{align*}
\end{theorem}
\begin{proof}
The little $q$-Legendre  polynomials $p_n(x|q) = \,_2\phi_1\left( \begin{array}{cccc}
                     q^{-n} , q^{n+1} \\
                    q
                   \end{array}
\mid q;qx\right)  $   satisfies the second-order $q$-difference equation, see \cite[Eq.(3.12.16)]{askey},
\begin{align}\label{kjhgd}
    \frac{1}{q} D_{q^{-1}} D_qy(x)+   \frac{qx+x-1}{qx(qx-1)} D_{q^{-1}}y(x) + \frac{ [n]_q [n+1]_q}{q^nx (1-qx)}y(x) =0.
\end{align}
By comparing  Equation (${\ref{kjhgd}}$) with Equation (${\ref{mxlk}}$),  we get
\begin{align*}
  p(x) = \frac{qx+x-1}{qx(qx-1)},\quad r(x) =   \frac{ [n]_q [n+1]_q}{q^nx (1-qx)} .
\end{align*}
Then $f(x) = x(1-qx)$ is a solution of Equation (${\ref{fht}}$).
From Equation(${\ref{xyy}}$), we get
\begin{align*}
  \frac{1}{q} D_{q^{-1}} D_qh(x)+  \frac{qx+x-1}{qx(qx-1)}D_{q^{-1}}h(x) +\frac{ [n]_q [n+1]_q}{q^nx (1-qx)} h(x) =\frac{1}{x(1-qx)}.
\end{align*}
Then  $h(x) = \frac{q^n}{[n]_q [n+1]_q}$.
From Theorem {\ref{x}},  we obtain
\begin{align}\label{251444}
  \int p_n(x|q) d_qx = \frac{-q^{n-1} }{[n]_q [n+1]_q }x (1-x) D_{q^{-1}} p_n(x|q).
\end{align}
 But
\begin{align}\label{4666}
  D_{q^{-1}}p_n(x|q)=-q^{-n} [n]_q [n+1]_q  \,_2\phi_1(q^{1-n} , q^{n+2};q^2;q,x).
\end{align}
Then Equation (${\ref{251444}}$) is
\begin{align*}
  \int p_n(x|q) d_qx = \frac{x }{q}(1-x) \,_2\phi_1(q^{-n+1} , q^{n+2};q^2;q,x).
\end{align*}
\end{proof}
\begin{theorem}\label{thmnb2}
  For $n \in \mathbb{N}$, we have
  \begin{align}\label{22145l}
  \int p_n(x;c;q) d_qx = \dfrac{q^{n-1}(q-x)(cq-x)}{[n]_q [n+1]_q(1-q)}\,_3\phi_2 ( q^{-n} , q^{n+1},x;q;cq;q,q).
\end{align}
\end{theorem}
\begin{proof}
The big $q$-Legendre  polynomials $p_n(x;c;q) = \,_3\phi_2\left( \begin{array}{cccc}
                     q^{-n} , q^{n+1},x \\
                    q,cq
                   \end{array}
\mid q;q\right)  $  satisfies the second-order $q$-difference equation, see \cite[Eq.(3.5.17)]{askey},
\begin{align}\label{kjhgdd}
    \frac{1}{q} D_{q^{-1}} D_qy(x)+  \dfrac{ x(1+q)-q(c+1)}{q^2  (x-1)(x-c)}D_{q^{-1}}y(x) - \frac{[n]_q [n+1]_q}{q^{1+n}(x-1)(x-c)} y(x) =0.
\end{align}
By comparing  Equation (${\ref{kjhgdd}}$) with Equation (${\ref{mxlk}}$),  we get
\begin{align*}
 p(x) = \dfrac{ x(1+q)-q(c+1)}{q^2  (x-1)(x-c)},\quad  r(x) =   - \frac{[n]_q [n+1]_q }{q^{1+n}(x-1)(x-c)}.
\end{align*}
Then $f(x) = (1-x)(1-\frac{x}{c})$ is a solution of Equation (${\ref{fht}}$).
From Equation (${\ref{xyy}}$), we get
\begin{align*}
  \frac{1}{q} D_{q^{-1}} D_qh(x)+  \dfrac{ x(1+q)-q(c+1)}{q^2  (x-1)(x-c)}D_{q^{-1}}h(x)  - \frac{[n]_q [n+1]_q }{q^{1+n}(x-1)(x-c)} h(x) =\frac{1}{(1-x)(1-\frac{x}{c})}.
\end{align*}
Hence, $ h(x) = \dfrac{c q^{n+1}}{[n]_q [n+1]_q}$.
From Theorem {\ref{x}},  we obtain
 \begin{align}\label{2465}
  \int p_n(x;c;q) d_qx = \dfrac{-q^{n-1}(q-x)(cq-x)}{[n]_q [n+1]_q}D_{q^{-1}}  p_n(x;c;q).
\end{align}
But
\begin{align*}
 D_{q^{-1}}  p_n(x;c;q) = \frac{-1}{1-q} \,_3\phi_2 ( q^{-n} , q^{n+1},x;q;cq;q,q).
\end{align*}
Then Equation (${\ref{2465}}$) can be written as
\begin{align*}
  \int p_n(x;c;q) d_qx = \dfrac{q^{n-1}(q-x)(cq-x)}{[n]_q [n+1]_q (1-q)}\,_3\phi_2 ( q^{-n} , q^{n+1},x;q;cq;q,q).
\end{align*}
\end{proof}
\begin{theorem}\label{gbhvsffsfs}
Let $\lambda$ and $\nu$ be complex numbers, $ \Re  (\nu) > \frac{-1}{2}$ and $C_\nu(q)$ is the constant  defined  in Theorem {$\ref{theorem4}$}. Then
\begin{align*}
 &C_\nu(q)\int x^{\nu}  J_\nu^{(3)}( \lambda x (1-q);q^2) d_qx  =\nonumber\\& \frac{ q^{\nu} x }{ \lambda^{{1+\nu} } } \left(  J_\nu^{(3)}( \frac{\lambda}{q} x (1-q) ;q^2) D_{q^{-1}} H_\nu^{(3)}(\lambda x;q^2)- H_\nu^{(3)}(\frac{\lambda  x}{q};q^2) D_{q^{-1}}  J_\nu^{(3)}( \lambda x (1-q);q^2)\right),
\end{align*}
and
\begin{align*}
 &  C_\nu(q)\int x^{\nu}  Y_\nu( \lambda x (1-q);q^2) d_qx \nonumber  =\\& \frac{q^{\nu} x }{ \lambda^{{1+\nu} } } \left(  Y_\nu( \frac{\lambda}{q} x (1-q);q^2) D_{q^{-1}}H_\nu^{(3)}(\lambda  x;q^2)- H_\nu^{(3)}(\frac{\lambda  x}{q};q^2) D_{q^{-1}}  Y_\nu( \lambda x (1-q);q^2)\right).
\end{align*}
\end{theorem}
\begin{proof}
Let $h(x) = H_\nu^{(3)}(\lambda x;q^2)$, then
  \begin{align}\label{hnhnhnhn}
&\frac{1}{q} D_{q^{-1}}D_q h(x) + \frac{1}{qx} D_{q^{-1}} h(x) +\Big(q^{-\nu-1} \lambda ^{2}  - \frac{q^{-\nu} [\nu]^2_q}{x^2}  \Big) h(x)=\dfrac{\lambda^{\nu+1} C_\nu(q)}{q^{\nu+1} }x^{\nu-1}.
\end{align}
 Clearly, $f(x) = x$ is a solution of Equation (${\ref{fht}}$). Substituting  with (${\ref{hnhnhnhn}}$) into Equation (${\ref{mkl}}$), we get the required result.
\end{proof}
  \begin{theorem}\label{jnhbgfertyui}
  Let $\lambda$ and $\nu$ be complex numbers, $ \Re  (\nu) > \frac{-1}{2}$ and $C_\nu(q)$ is the constant  defined  in Theorem {$\ref{theorem4}$}. Then
\begin{align*}
&C_\nu(q)\int \dfrac{x^{\nu}} {(-x^2  \lambda^2(1-q)^2;q^2)_\infty} J_\nu^{(2)}( \lambda x |q^2) d_qx  =\nonumber \\& \frac{ \lambda^{-1-\nu}q^{\nu} x  }{(-x^2 \lambda^2(1-q)^2;q^2)_\infty }\left(  J_\nu^{(2)}( \lambda x |q^2) D_{q^{-1}}H_\nu^{(2)}(\lambda x|q^2)- H_\nu^{(2)}(\lambda x|q^2) D_{q^{-1}}  J_\nu^{(2)}( \lambda x |q^2)\right).
\end{align*}
\end{theorem}
\begin{proof}
Let $h(x) = H_\nu^{(2)}(\lambda x|q^2)$, then
  \begin{align}\label{jmjmslsl}
\frac{1}{q} D_{q^{-1}}D_q h(x) + \left(\dfrac{1}{x}-q \lambda^2 x (1-q) \right)D_{q} h(x) +\left(q  \lambda ^{2}  - \frac{q^{1-\nu} [\nu]^2_q }{x^2} \right)h(x)=\dfrac{\lambda^{\nu+1} C_\nu(q) }{q^{\nu}}x^{\nu-1}.
\end{align}
The function $f(x) = \dfrac{x}{(- \lambda^2(1-q)^2 x^2 ;q^2)_\infty}$ is a solution of Equation (${\ref{fht}}$).
Substituting with (${\ref{jmjmslsl}}$) into Equation (${\ref{mkrl}}$), we get the desired result.
\end{proof}
\vskip0.5cm
\begin{theorem}
 Let $\lambda$ and $\nu$ be complex numbers, $ \Re  (\nu) > \frac{-1}{2}$ and $C_\nu(q)$ is the constant  defined  in Theorem {$\ref{theorem4}$}. Then
  \begin{align*}
 &C_\nu(q)\int x^{\nu} (-x^2 \lambda^2(1-q)^2;q^2)_\infty J_\nu^{(1)}( \lambda x |q^2) d_qx  =\nonumber \\& \frac{q^{\nu} x }{ \lambda^{{1+\nu} }}   (-\frac{\lambda^2}{q^2} (1-q)^2 x^2 ;q^2)_\infty \left(  J_\nu^{(1)}( \frac{\lambda x}{q} |q^2) D_{q^{-1}}H_\nu^{(1)}(\lambda x|q^2)- H_\nu^{(1)}(\frac{\lambda x}{q}|q^2) D_{q^{-1}}  J_\nu^{(1)}( \lambda x |q^2)\right).
\end{align*}
\end{theorem}
\begin{proof}
  The proof is similar to the proof of Theorem {$\ref{gbhvsffsfs}$} and Theorem {$\ref{jnhbgfertyui}$}, and is omitted.
\end{proof}

\noindent {\bf Acknowledgements} \\
Authors are thankful to the learned referees for their valuable comments which improved the presentation of the paper.
\\
\\
\noindent {\bf Funding}\\
Not applicable
\\
\\
\noindent {\bf Availability of data and materials}\\
The data and material in this paper are original.
\\
\\
\noindent {\bf Competing interests} \\
The authors declare that they have no competing interests.
\\
\\
\noindent{\bf Author's contributions}\\
GH, ZM, and KO together studied and prepared the manuscript. ZM and KO analyzed all the results and made necessary
improvements. GH is the major contributor in writing the paper. All authors read and approved the final manuscript.

\end{document}